\documentclass[11pt,twoside, a4paper]{article}
\usepackage{latexsym}
\usepackage[english]{babel}
\usepackage{t1enc}
\usepackage{mathrsfs}
\usepackage{ifthen}
\usepackage{url}
\usepackage{epsfig}
\usepackage{amsbsy}
\usepackage{extarrows}
\usepackage{amsfonts}
\usepackage{amsmath}
\usepackage{amssymb}
\usepackage{amsthm}
\usepackage{amsxtra}
\usepackage{bbm}
\usepackage{color}
\usepackage{relsize} 
\usepackage{enumitem}
\usepackage{graphicx}
\usepackage{caption}
\usepackage{subcaption}
\usepackage{placeins}
\usepackage{wrapfig}
\usepackage{float}

\usepackage[margin=2.5cm]{geometry}

\usepackage{verbatim}

\numberwithin{equation}{section}
\numberwithin{figure}{section}

\usepackage{mathptmx}

\theoremstyle{plain} 
\newtheorem{theorem}{Theorem}[section]
\newtheorem*{theorem*}{Theorem}
\newtheorem{lemma}[theorem]{Lemma}
\newtheorem*{lemma*}{Lemma}
\newtheorem{corollary}[theorem]{Corollary}
\newtheorem*{corollary*}{Corollary}
\newtheorem{proposition}[theorem]{Proposition}
\newtheorem*{proposition*}{Proposition}

\newtheorem*{definition*}{Definition}

\newtheorem*{example*}{Example}

\newtheorem*{remark*}{Remark}
\newtheorem*{remarks*}{Remarks}
\newtheorem*{notation}{Notation}

\theoremstyle{definition}
\newtheorem{definition}[theorem]{Definition}

\theoremstyle{remark}
\newtheorem{remark}[theorem]{Remark}

\definecolor{olivegreen}{rgb}{0,0.6,0.1}

\newcommand{\1} {\mspace{1 mu}}
\newcommand{\2} {\mspace{2 mu}}
\newcommand{\msp}[1] {\mspace{#1 mu}}

\newcommand{\eps}{\varepsilon}

\newcommand{\R} {\mathbb{R}}
\newcommand{\C} {{\mathbb{C}}}
\newcommand{\N} {{\mathbb{N}}}
\newcommand{\Cp} {\mathbb{H}}

\newcommand{\Bounded}{\mathcal{B}}

\newcommand{\xSpace} {\mathfrak{X}} 
\newcommand{\xMeasure} {\pi} 

\newcommand{\ord}{\mathcal{O}}

\newcommand{\E} {\mathbb{E}}

\newcommand{\ee}{\mathrm{e}} 
\newcommand{\ii}{\mathrm{i}} 
\newcommand{\dd}{\mathrm{d}}

\newcommand{\p}[1]{({#1})}
\newcommand{\pb}[1]{\bigl({#1}\bigr)}
\newcommand{\pB}[1]{\Bigl({#1}\Bigr)}
\newcommand{\pbb}[1]{\biggl({#1}\biggr)}

\newcommand{\sB}[1]{\Bigl[{#1}\Bigr]}

\newcommand{\abs}[1]{\lvert #1 \rvert}
\newcommand{\absb}[1]{\big\lvert #1 \big\rvert}
\newcommand{\absB}[1]{\Big\lvert #1 \Big\rvert}
\newcommand{\absbb}[1]{\bigg\lvert #1 \bigg\rvert}

\newcommand{\norm}[1]{\lVert #1 \rVert}

\newcommand{\avg}[1]{\langle #1 \rangle}
\newcommand{\avgb}[1]{\big\langle #1 \big\rangle}
\newcommand{\avgB}[1]{\Big\langle #1 \Big\rangle}

\newcommand{\scalar}[2]{\langle{#1} \mspace{2mu}, {#2}\rangle}

\DeclareMathOperator{\re}{Re}
\DeclareMathOperator{\im}{Im}

\usepackage{hyperref}

\begin{document}
\renewcommand{\thefootnote}{\fnsymbol{footnote}}
\title{\bf Singularities of solutions to quadratic vector equations on complex upper half-plane}

\author{
\begin{minipage}{0.3\textwidth}
 \begin{center}
Oskari H. Ajanki\footnotemark[1]\\
\footnotesize 
{IST Austria}\\
{\url{oskari.ajanki@iki.fi}}
\end{center}
\end{minipage}
\begin{minipage}{0.3\textwidth}
\begin{center}
L\'aszl\'o Erd{\H o}s\footnotemark[2]  \\
\footnotesize {IST Austria}\\
{\url{lerdos@ist.ac.at}}
\end{center}
\end{minipage}
\begin{minipage}{0.3\textwidth}
 \begin{center}
Torben Kr\"uger\footnotemark[3]\\
\footnotesize 
{IST Austria}\\
{\url{torben.krueger@ist.ac.at}}
\end{center}
\end{minipage}
}

\date{} 

\maketitle
\thispagestyle{empty} 
	
\footnotetext[1]{Partially supported by ERC Advanced Grant RANMAT No.\ 338804, and SFB-TR 12 Grant of the German Research Council.}
\footnotetext[2]{Partially supported by ERC Advanced Grant RANMAT No.\ 338804.}
\footnotetext[3]{Partially supported by ERC Advanced Grant RANMAT No.\ 338804, and SFB-TR 12 Grant of the German Research Council}

\maketitle
\thispagestyle{empty} 
\vspace{-0.5cm}

\begin{abstract}
Let $ S $ be a positivity preserving symmetric linear operator acting on bounded functions.
The nonlinear equation $-\frac{1}{m}=z+Sm$ with a parameter $z$ in the complex upper half-plane $ \mathbb{H} $ has a unique solution $ m  $ with values in $\mathbb{H}$.
We show that the $ z $-dependence of this solution can be represented as the Stieltjes transforms of a family of probability measures $v$ on $\mathbb{R}$. Under suitable conditions on $ S $, we show that $v$ has a real analytic density apart from finitely many algebraic singularities of degree at most three.

Our motivation  comes from large random matrices. The solution $m$ determines the density of eigenvalues of two prominent matrix ensembles; (i) matrices  with centered independent entries  whose variances are given by $S$ and  (ii) matrices with correlated entries with a translation invariant correlation structure. Our analysis shows
that the limiting eigenvalue density  has only square root 
singularities or a cubic root cusps; no other singularities occur.
\end{abstract}
\vspace{0.2cm}
{\bf AMS Subject Classification (2010):} \texttt{45Gxx}, \texttt{46Txx}, \texttt{60B20}, \texttt{15B52}.


\section{Introduction}
\label{sec:Introduction}
\pagenumbering{arabic}
\setcounter{page}{1}

Given a symmetric $N\times N$-matrix $S=(s_{ij})_{i,j=1}^N$ with non-negative entries and a complex number $z$ in the upper half-plane $\mathbb{H} := \{z\in \mathbb{C} :\, \mathrm{Im} \,z\,>\,0\} $, we consider the system of $N$ non-linear equations
\begin{equation}
\label{NSCE}
 -\,\frac{1}{m_i\!}\;=\;z\,+\,\sum_{j=1}^N s_{ij}\1 m_j\, , \qquad  i =1,\dots, N\,,
\end{equation}
for $N$ unknowns $ m_1,\dots,m_N \in \mathbb{H}$. This is one of the simplest nonlinear systems of equations involving a general linear part. 
We will consider a general version of this problem, where $ S $ is replaced by a linear operator acting on the Banach space of  $\mathbb{H}$-valued bounded functions  $m$ on some measure space that replaces the underlying discrete space of $N$ elements in \eqref{NSCE}.

In this paper we give a detailed analysis of the  singularities of
the solution of \eqref{NSCE} as a function of the parameter $z$. In particular, we show
that $m(z)$ is analytic in $z$ down to the real axis, apart from a few singular points. Our main result, Theorem~\ref{thr:Regularity and singularities of the generating density}, asserts that, under some natural conditions on $S$, these singularities are algebraic and they can be  of degree two or three only. In fact, the components of $m(z)$ are the Stieltjes transforms of a family of densities on $\mathbb{R}$ 
with a common support consisting of finitely many compact intervals. The singularities of $m(z)$ originate from the  behavior of these densities at the points where they  vanish. We show that the densities are real analytic whenever they are positive and 
that they can only have square root singularities at the boundary (edges) of their support and cubic root  cusps inside the interior.

Looking at \eqref{NSCE} directly,  the solutions of this system of $N$ quadratic equations in $N$ variables are algebraic functions of large degree in $z$. Thus, algebraic singularities of very high order are  theoretically possible and in case of an infinite dimensional Banach space apparently even  non-algebraic singularities might emerge. Our result excludes these scenarios and precisely classifies the possible singularities.  

The system of equations \eqref{NSCE} naturally appears in the spectral analysis of large random matrices and other related problems. It has been analysed in this context by many authors; e.g. Berezin \cite{Ber}, Wegner \cite{WegnerNorb},  Girko \cite{Girko-book}, Khorunzhy and Pastur \cite{KhorunzhyPastur94}, see also the more recent papers \cite{AZdep,Guionnet-GaussBand,ShlyakhtenkoGBM}.
We mention two prominent examples in this direction.

The first example is a Wigner-type matrix with a   general variance structure (Section~\ref{sec:wigner}).
Let $ H=(h_{ij})_{i,j=1}^N $ be an $N\times N$ real symmetric or complex hermitian matrix with centered entries and variance matrix $ S$, i.e., $\mathbb{E}\,h_{ij}=0$, $ \mathbb{E}\mspace{2 mu}|h_{ij}|^2 = s_{ij}$. The matrix elements are independent up to the symmetry constraint, $ h_{ji}= \overline{h}_{ij} $. 
Let $ G(z)=(H-z)^{-1}$ be the resolvent of $H$ with a spectral parameter $ z \in \mathbb{H} $  and  matrix elements   $g_{ij} = g_{i j}(z)$. 
Second order perturbation theory indicates that
\begin{equation}
\label{Gii}
- \frac{1}{g_{ii}\!} \,\approx\,  z + \sum_{j=1}^N s_{ij} \1g_{j\msp{-1}j}
\,,
\end{equation}
(see \cite{EYY} in the special case when the sums $\sum_j s_{ij}$ are independent of $i$).
In particular, if \eqref{NSCE} is stable under small perturbations, then
   $ g_{ii} $  is close to $ m_i $ and the average $N^{-1}\sum_i m_i$ approximates the normalized trace of the resolvent, $ N^{-1}\mathrm{Tr}\, G$.
Being determined by $N^{-1} \mathrm{Im}\,\mathrm{Tr}\, G$, as $ \mathrm{Im}\,z \to 0 $, the empirical spectral measure of  $H$
approaches the non-random measure with density
\begin{equation}
\label{DOS}
\rho(\tau) \,:=\, 
\lim_{\eta\mspace{1 mu}\downarrow\mspace{1 mu} 0}
\frac{1}{\pi\1 N}
\sum_{i=1}^N
\mathrm{Im}\,m_i(\tau+\mathrm{i}\mspace{1 mu}\eta)
\,, \qquad \tau\in \mathbb{R}
\,,
\end{equation}
as $ N $ goes to infinity, see \cite{AZind,Guionnet-GaussBand,ShlyakhtenkoGBM}.

In the second example, we consider a random matrix of the form ${H}= {X}+{X}^*$ where 
  ${X}$ has {\it correlated entries} with a translation invariant correlation structure. 
As the dimension of $ {H} $ grows, its empirical spectral measure approaches  \cite{AEK3,BMP2013, Chakrabarty2014,KhorunzhyPastur94,PasturShcerbinaAMSbook} a non-random measure with density $ \rho $ defined through \eqref{DOS}.
In this setup $ m_i $ solves \eqref{NSCE} with $s_{ij}$ given by the Fourier transform of the correlation matrix 
(see Section~\ref{sec:trinv}).

Apart from a few specific cases, solving \eqref{NSCE} and  then computing $\rho(\tau)$ via \eqref{DOS}
is the only known effective way to determine the spectral density of these  large random matrices. Similarly, the limiting density, as  $N\to\infty$, is computed by solving a continuous (integral) version of \eqref{NSCE}.
Since much numerical and theoretical research  \cite{HachemCovSurvey2015,RaoPolyDOS,ChoiSiverstein1995}
focuses  on the density,
it is somewhat surprising that no detailed analytical study of \eqref{NSCE}  has been initiated so far which goes beyond establishing existence, uniqueness, and regularity in the regime where the parameter $z$ is away from the real axis \cite{AC-yau-students,Girko-book,Helton2007-OSE,PasturShcerbinaAMSbook}.
Typically, the limiting density is compactly supported  on a few disjoint intervals. 
Numerical studies \cite{BurdaVARMA,ErgunModGraph,KuhnModGraph,Menon-IEEEplots,RaoPolyDOS} indicate that generically the limiting density exhibits square root singularities at the edges of these intervals. However, prior to the current work this finding has never been rigorously confirmed apart from a few explicitly computable cases \cite{HachemCovSurvey2015,RaoPolyDOS,ChoiSiverstein1995}. 
Even less has been known about the possible formation of other singularities \cite{RaoPolyDOS}.
In a special Gaussian model the cubic singularity has been shown to emerge \cite{BH} as a gap in the support of the density closes. 
For random matrices with translation invariant correlation structure Theorem 2.4 of \cite{AZdep} shows that also the Stieltjes transform $m(z)$ of the limiting density satisfies a polynomial equation of the form $ P(\1z\1, m(z))=0$, but the degree of $ P $ is unspecified. Our result applies to the setup of \cite{AZdep} and limits the algebraic singularities to
 degrees two or three (Section~\ref{sec:color}).

 We remark that for invariant random matrix ensembles 
with a real analytic potential the singularities of the density have been classified \cite{DKMcL,KMcL,PasturSpecProb}. In the generic case the singularities at the edges of the support of the density are also of square root type.  For specific potentials the density may vanish at half-integer powers at the edges and any even power in the interior of the support
  but, in contrast to our result,  cubic root cusp singularities do not occur
for these ensembles.

So far we discussed {\it qualitative} properties of the spectral density on the macroscopic scale but 
the significance of \eqref{NSCE} goes well beyond that. First,
the algebraic order of the singularity of the limiting density at the edges predicts the typical scale of 
fluctuation of the extreme eigenvalues.
Second,
the recent proofs of the Wigner-Dyson-Mehta conjecture on the universality of local eigenvalue statistics heavily rely on  understanding the spectral density on very small scales (see \cite{EYBull} for a complete history).
 One key ingredient  to obtain such a  {\it local law} is a very accurate stability analysis   of the equation that  determines the spectral density, here \eqref{NSCE}. The stability deteriorates near the singularities
 and extracting the necessary information requires a thorough {\it quantitative} analysis of $m$ at these points.

Our project has three main parts. In the current paper we present general qualitative results on the singularities of  \eqref{NSCE}.  We believe that this analysis is of interest in its own  right since \eqref{NSCE} appears in other contexts as well, even beyond random matrices \cite{BolleNeri,KLW2,WegnerNorb}. The
quantitative description of the singularities together with the
 stability analysis, tailored to the specific needs to prove the local law, 
will be presented in details in \cite{AEK1}. Finally, 
the local law and spectral universality for the corresponding random matrices will be proved in  \cite{AEK2}
with a separate discussion of the correlated case in \cite{AEK3}.
The current paper is self-contained and random matrices will not
appear here; they were mentioned only to motivate the problem.

\section{Main result}\label{sec:main}

For a measurable space $\mathfrak{A}$ and a subset $\mathbb{D}\subseteq \C$ of the complex numbers, we denote by $\Bounded(\mathfrak{A},\mathbb{D})$ the space of bounded measurable functions on $\mathfrak{A}$ with values in $\mathbb{D}$. 
Let  $(\xSpace,\xMeasure(\dd x))$ be a measure space with bounded positive (non-zero) measure $\pi$. Suppose we are given a real valued  function $a\in \Bounded(\xSpace, \R)$ and a non-negative, symmetric, $s_{xy}=s_{yx}$, function $s\in \Bounded(\xSpace^2, \R_0^+)$. Then we consider the  \emph{quadratic vector equation (QVE)},
\begin{equation}
\label{QVE}
-\frac{1}{m(z)}\,=\, z+a+Sm(z)  \,,\qquad  z \in \Cp\,,
\end{equation}
for a function $m:\Cp \to \Bounded(\xSpace, \Cp), \, z \mapsto m(z)$, where $S:\Bounded(\xSpace, \C)\to \Bounded(\xSpace, \C)$ is the integral operator with kernel $s$,
\[
(Sw)_x\,:=\, \int s_{x y} \1w_y\1\xMeasure(\dd y)\,, \qquad x \in \xSpace\,,\; w \in \Bounded(\xSpace, \C)\,.
\]

We equip the space $\Bounded(\xSpace, \C)$ with its natural norm,
\[
\norm{w}\,:=\, \sup_{x \1\in\1\xSpace}\2\abs{w_x}\,, \qquad w \in \Bounded(\xSpace,\C)\,.
\]
With this norm $\Bounded(\xSpace, \C)$ is a Banach space. 
For an operator $T$ on $ \Bounded(\xSpace,\C)$ we write $\norm{T}$ for the induced operator norm. 

 The following result is considered folklore in the literature (see e.g. \cite{AZind,Girko-book,Helton2007-OSE,KLW2,KhorunzhyPastur94}). For completeness we include its proof, adjusted to our setup, in the  Appendix~\ref{appendix:Existence and uniqueness}.

\begin{proposition}[Existence and uniqueness]
\label{prp:Existence and uniqueness}
The QVE has a unique solution $m$. 
For each $ x \in \xSpace $ there exists a unique probability measure $ v_x(\dd \tau ) $ on $ \R $ such that
\begin{equation}
\label{m as stieltjes transform}
m_x(z) \,=\,\int_{\R} \frac{\1v_x(\dd \tau)\!}{\tau\1-\1z}
\,, 
\qquad   z \in \Cp 
\,.
\end{equation}
All these measures have support in the compact interval $[-\kappa,\kappa]$ with 
\begin{equation}
\label{definition of kappa}
\kappa \,:=\, \norm{a} + 2 \1\norm{S}^{1/2}.
\end{equation}
The family $(v_x)_{x \in \xSpace}$ constitutes a measurable function $ v : \xSpace \to \mathcal{M}(\R),\; x \mapsto v_x $, where $\mathcal{M}(\R) $ denotes the space of probability measures on $ \R $ equipped with the weak topology.
\end{proposition}

Since by \eqref{m as stieltjes transform} the measure $v_x$ determines the component $m_x$ of the solution to the QVE, we consider this measure as a fundamental quantity and often express properties of $m$ in terms of properties of $v$. 

\begin{definition}[Generating measure]
The family of probability measures $v=(v_x)_{x \in \xSpace}$, uniquely defined by the relation \eqref{m as stieltjes transform}, is called the \emph{generating measure} of the QVE. In case all measures $v_x$ admit a bounded Lebesgue-density, $v_x(\dd \tau)=v_x(\tau)\dd \tau$, the (almost everywhere defined) function $v: \R \to \Bounded(\xSpace,\R_0^+),\, \tau \mapsto v(\tau) $  is called the  \emph{generating density} of the QVE.
\end{definition}
By a trivial rescaling we may from now on assume that $\pi$ is a probability measure, $\xMeasure(\R)=1$. We introduce a short notation for the average of a function on $\xSpace$ with respect to this measure
\[
\avg{\1w\1}\,:=\, \int w_x \1\xMeasure(\dd x)\,, \qquad w \in \Bounded(\xSpace,\C)\,.
\]
We need three assumptions on the data $ a $ and $s$ of the QVE. 

\begin{enumerate}
\item[{\bf (A)}] \emph{Diagonal positivity}: There exists a symmetric, $ r_{yx} = r_{xy} $, function  $r\in \Bounded(\xSpace^2,\R_0^+)$ such that 
\begin{equation}
\label{positive diagonal}
s_{x y}\,\geq\, \int r_{x u}\1r_{u y}\1\xMeasure(\dd u)\,, \qquad x,y \in \xSpace\,,
\end{equation}
and $\inf_{x \in \xSpace}\int r_{x y} \1\xMeasure(\dd y)>0$.
\item[{\bf (B)}] \emph{Uniform primitivity}: There is some $K \in \N$ such that
\begin{equation}
\inf_{x,y \in \xSpace}s^{(K)}_{x y}\,>\,0\,,
\end{equation}
where $s^{(K)}\in \Bounded(\xSpace^2,\R_0^+)$ is the integral kernel of the $K$-th power $S^{\1K}$ of the operator $S$.
\item[{\bf (C)}] \emph{Component regularity}:  There are no outlier components in the sense that
\begin{equation}
\label{component regularity}
\lim_{\eps\downarrow 0}\inf_{x \in \xSpace}\int \frac{\xMeasure(\dd y)}{\eps+(a_x-a_y)^2+\avg{\1(S_x-S_y)^2}}\,=\, \infty \,,
\end{equation}
where $ S_x := (y \mapsto s_{x y}) \in \Bounded(\xSpace, \R_0^+)$ for each $ x \in \xSpace$ denotes the $ x$-th component of $ S $.
\end{enumerate}

In the following remarks we mention a few examples that illustrate the meaning of the assumptions \textbf{(A)}, \textbf{(B)} and \textbf{(C)}. The statements of these remarks are easy to check. 
\begin{remark} If $\xSpace$ is equipped with a metric $d_\xSpace$ and $s$ has a positive strip along the diagonal in the sense that
\[
s_{x y}
\,\geq\,c\1\mathbbm{1}(d_\xSpace(x,y)\leq \eps)\,,\qquad x,y \in \xSpace\,,
\]
for some positive constants $c$ and $\eps$, then \eqref{positive diagonal} is satisfied with the choice
\[
r_{x y}\,:=\, c^{1/2}\mathbbm{1}(d_\xSpace(x,y)\leq \eps/2)\,,\qquad x,y \in \xSpace\,.
\]
\end{remark}

\begin{remark} If $\xSpace$ admits a finite measurable partition, $\mathfrak{I}_1, \dots, \mathfrak{I}_n$, such that each $\mathfrak{I}_i$ has positive measure and there is a primitive matrix $P=(p_{i j})_{i,j=1}^{n} \in \R^{n \times n}$ such that $s$ satisfies
\[
s_{x y}\,\geq\, \sum_{i,j=1}^{n}p_{i j}\mathbbm{1}(x \in \mathfrak{I}_i\1,\2 y \in \mathfrak{I}_j)\,,
\qquad x,y \in \xSpace\,,
\]
then \textbf{(B)} holds. Recall that a matrix $P$ is called \emph{primitive}, if it has non-negative entries and there exists $k \in \N$ such that all entries of $P^k$ are strictly positive. If, additionally, the matrix $P$ has a positive diagonal, $p_{ii}>0$, then assumption \textbf{(A)} is satisfied as well.
\end{remark}

\begin{remark}  Assumption \textbf{(C)} is trivially satisfied if $\xSpace$ is finite and every $x \in \xSpace$ has positive measure.
The integral in \eqref{component regularity} provides a way to measure the regularity of $s$ and $a$ as the following example indicates. 
Consider the setup $(\xSpace, \xMeasure(\dd x))=([0,1], \dd x) $ and let $I_1, \dots, I_n$ be a partition of $[0,1]$ into finitely many non-trivial intervals.  Suppose  $a$ and $s$ are piecewise $1/2$-H\"older continuous,
\begin{equation}
\label{piecewise Holder}
\sup_{x,\1y \1\in\1 I_i}\frac{\abs{a_x-a_y}}{\msp{10}\abs{x-y}^{1/2}\msp{-5}}\,<\, \infty
\,,\quad 
\sup_{\genfrac{}{}{0pt}{1}{x_1,\2x_2 \in I_i}{y_1,\2y_2 \in I_j}}\frac{\abs{s_{x_1 y_1}-s_{x_2 y_2}}}{\abs{x_1-x_2}^{1/2}+\abs{y_1-y_2}^{1/2}}\,<\, \infty
\,,
\end{equation}
for all $i,j=1,\dots,n$.
Then \textbf{(C)} is satisfied.
\end{remark}

Now we state our main theorem. 

\begin{theorem}[Regularity and singularities of the generating density]
\label{thr:Regularity and singularities of the generating density}
Suppose $s$ and $a$ satisfy the assumptions \textbf{(A)}, \textbf{(B)} and  \textbf{(C)}. Then the generating measure has a Lebegue-density $v_x(\dd \tau)=v_x(\tau)\dd \tau$ and this density is uniformly $1/3$-H\"older continuous
\[
\sup_{x \in \xSpace}\sup_{\tau_1\neq \tau_2}\frac{\abs{\1v_x(\tau_1)-v_x(\tau_2)}}{\abs{\tau_1-\tau_2}^{1/3}}\,< \, \infty\,.
\]
The set on which the $x$-th component of the generating density is positive is independent of $x$ (and therefore the component is not indicated),
\begin{equation}
\label{definition of set S}
\mathfrak{S}\,:=\,\bigl\{\tau \in \R:\; v_x(\tau)\1>\10\1\bigr\}
\,.
\end{equation}
It  is a union of finitely many open intervals.
The restriction of the generating density $v: \R\backslash\partial \mathfrak{S} \to \Bounded(\xSpace,\R_0^+), \tau \mapsto v(\tau)$ is analytic in $\tau$. 
At the points $\tau_0 \in \partial \mathfrak{S}$ the generating density has one of the following two behaviors:
\begin{description}
\item[Cusp] If $\tau_0$ is in the intersection of the closure of two connected components of $\mathfrak{S}$, then $v$ has a cubic root singularity at $\tau_0$, i.e., there is some $c\in \Bounded(\xSpace, \R)$ with $\inf_{x \in \xSpace}c_x>0$ such that uniformly in $x\in \xSpace$,
\begin{equation}
\label{cusp singularity}
v_x(\tau_0+ \omega)\,=\, c_x\2\abs{\omega}^{1/3}+\ord(\abs{\omega}^{2/3})\,, \quad \omega \to 0\,.
\end{equation}
\item[Edge] If $\tau_0 \in \partial \mathfrak{S}$ is not a cusp, then it is the left or right endpoint of a connected component of $\mathfrak{S}$ and $v$ has a square root singularity at $\tau_0$, i.e., there is some $c\in \Bounded(\xSpace, \R)$ with $\inf_{x \in \xSpace}c_x>0$ such that uniformly in $x\in \xSpace$,
\begin{equation}
\label{edge singularity}
v_x(\tau_0\pm \omega)\,=\, c_x\2\omega^{1/2}+\ord(\omega)\,, \quad \omega \downarrow 0\,,
\end{equation}
where $\pm$ is taken depending on whether $\tau_0$ is a left or a right endpoint. 
\end{description}
\end{theorem}

Let us denote the extended upper half plane by $ \overline{\Cp} :=\Cp \cup\R $.

\begin{corollary}[H\"older-regularity of the solution]
\label{crl:Holder-regularity of the solution}
Assume  that $s$ and $a$ satisfy \textbf{(A)}, \textbf{(B)} and  \textbf{(C)}. Then
the solution $m_x(z)$ of the QVE can be uniquely extended to a function $m \in\Bounded(\xSpace\times \overline{\Cp}, \overline{\Cp}) $ that is uniformly $1/3$-H\"older continuous in its second argument, $z \in \overline{\Cp}$,
\begin{equation}
\label{Holder continuity of m}
\sup_{x \in \xSpace}\sup_{z_1\neq z_2}\frac{\abs{\1m_x(z_1)-m_x(z_2)}}{\abs{z_1-z_2}^{1/3}}\,< \, \infty\,.
\end{equation}
\end{corollary}

\begin{theorem}[Single interval support]
\label{thr:Single interval support}
Assume \textbf{(A)}, \textbf{(B)} and \textbf{(C)}. Furthermore, suppose that the components $a_x$ and $S_x=(y \mapsto s_{xy}) \in \Bounded(\xSpace,\R_0^+)$ form a connected set in the sense that
\begin{equation}
\label{connected rows}
\sup_{\emptyset\neq\1\mathfrak{A} \subsetneq \xSpace}\inf_{\genfrac{}{}{0pt}{1}{x \in \mathfrak{A}}{y \not \in \mathfrak{A}}}\;
\pb{\2\abs{a_x-a_y}+\avg{\1\abs{S_x-S_y}\1}}\,=\, 0\,.
\end{equation}
Then the set $\mathfrak{S}$, defined in \eqref{definition of set S}, is a single interval. In particular, there are no cusp singularities in the sense of \eqref{cusp singularity}.
\end{theorem}

\begin{remark} Condition \eqref{connected rows} is for example satisfied in the special case when $\xSpace=[0,1]$ and the kernel $s$  as well as $a$ are  continuous.  More generally, we may consider the piecewise H\"older continuous setting. We conjecture that in this case with $a=0$ the number of connected components of $\mathfrak{S}$ is at most $2n-1$, where $n$ is the size of the partition $I_1, \dots, I_n$ used to define the piecewise H\"older continuity in \eqref{piecewise Holder}. 
\end{remark}

\section{Applications}

In this section we discuss four applications in random matrix theory.  We denote by $H=H^{(N)}$ for $N \in \N$, a sequence of self-adjoint random matrices with entries $h_{i j}=h_{i j}^{(N)}$ on some probability space with expectation  $\mathbb{E} $. 
We define the induced normalized empirical spectral measures by
\[
\rho^{(N)}\msp{-2}(B)\,:=\,\frac{\abs{\1\mathrm{Spec}(\1H^{(N)}) \cap\, B\1}}{N}
\,, 
\]
for Borel sets $ B $ of $  \R $.
We denote by $\rho_{S\msp{-1},a}(\dd \tau)$ the average generating measure for the QVE with operator $ S $ and function $ a $, i.e., 
\[
\rho_{S\msp{-1},a}(\dd \tau)\,:=\, \int_{\xSpace} v_x(\dd \tau)\2\xMeasure(\dd x)\,.
\]

\subsection{Wigner type matrices}\label{sec:wigner}

A natural extension of Wigner matrices, which have i.i.d. entries up to symmetry constraints, are what  was called \emph{Wigner type} matrices  in \cite{AEK2}. These matrices 
 $H$
are self-adjoint, centered, $\E\2h_{i j}=0$, and have independent entries up to the symmetry constraints, i.e., $h_{i j}$ are independent for $i \leq j$. 
Furthermore, let $ N\,\abs{h_{ij}}^2 $ be uniformly integrable. 
Suppose for the sake of simplicity that the variances of the entries of $ H $ converge  to a piecewise $1/2$-H\"older continuous (cf. \eqref{piecewise Holder}), symmetric, $q(x,y)=q(y,x) $, profile function $q:[0,1]^2\to \R^+_0 $ with a non-vanishing diagonal, $\inf_{\abs{x-y}\leq \eps}q(x,y)>0$ for some $\eps>0 $, i.e., 
\[
\E\2 \abs{\1h_{i j}}^2\,=\, \frac{1}{N}\,q\Big(\frac{i}{N} , \frac{j}{N}\Big)
\,.
\]
Then the empirical spectral measures of the matrices $H^{(N)}$ converge, as $N \to \infty$, to a non-random limit,
\begin{equation}
\label{convergence of rhoN to rhoS}
\rho^{(N)}(\dd \tau)\,\to\, \rho_S(\dd \tau)\,,\qquad \text{weakly, in probability}
\,.
\end{equation}
Here, the asymptotic spectral measure $\rho_S = \rho_{S,0}$ is obtained from the solution of the QVE in the setup $ (\xSpace,\xMeasure(\dd x)) = ([0,1],\dd x)$, where $a=0$, and the integral kernel of $ S $ is given by the asymptotic variance profile $s_{x y}:=q(x,y)$. 
For a proof of \eqref{convergence of rhoN to rhoS} see \cite{Guionnet-GaussBand,ShlyakhtenkoGBM} (in the Gaussian setting),  \cite{AZind} (with the additional condition that the fourth moments have a profile), and \cite{AEK2} (with bounded higher moments).
Bounded moment conditions can be relaxed using a standard cut-off argument (cf. Theorem 2.1.21 in \cite{IntroductionRM}). 

Using Theorem \ref{thr:Regularity and singularities of the generating density}, we can say more about the limiting eigenvalue distribution $\rho_S$. In fact, $\rho_S(\dd \tau)=\rho_S(\tau)\1 \dd \tau$ has a H\"older-continuous density  with singularities of degree at most three in the sense of \eqref{cusp singularity} and \eqref{edge singularity}. Moreover, if $q$ is $1/2$-H\"older continuous (not just piecewise H\"older continuous), then by Theorem \ref{thr:Single interval support} the limiting spectral density $\rho_S$ is supported on a single interval $[-\tau_0,\tau_0]$ and has square root singularities at the edges $-\tau_0$ and $\tau_0$.

In general, cusps may appear  already in the simplest non-trivial examples. This is illustrated by the $2\times 2$-block profile 
\begin{equation}
\label{profile function}
q\;:=\; \alpha\1\mathbbm{1}_{I \times I}+\beta\2(\mathbbm{1}_{I \times I^c}+\mathbbm{1}_{I^c \times I})+\gamma\2\mathbbm{1}_{I^c \times I^c}\,,
\end{equation}
where $\alpha,\beta,\gamma$ are positive constants, $I=[0,\delta]$ and $I^c=(\delta,1]$ for some $\delta \in (0,1/2]$. For example, in the special case $\beta=1$ and $\gamma=1/\alpha$ the choice $\delta=\delta_c(\alpha)$ with $\alpha>2$ leads to a density $\rho_S$ with a cusp singularity (cf. Fig. \ref{Fig:2by2Cusp}), where
\[
\delta_c(\alpha)\,:=\, \frac{(\alpha-2)^3}{9\1(\alpha^3-2\alpha^2+2\alpha-1)}\,.
\]
\begin{figure}[h]
	\centering
	\includegraphics[width=0.8\textwidth]{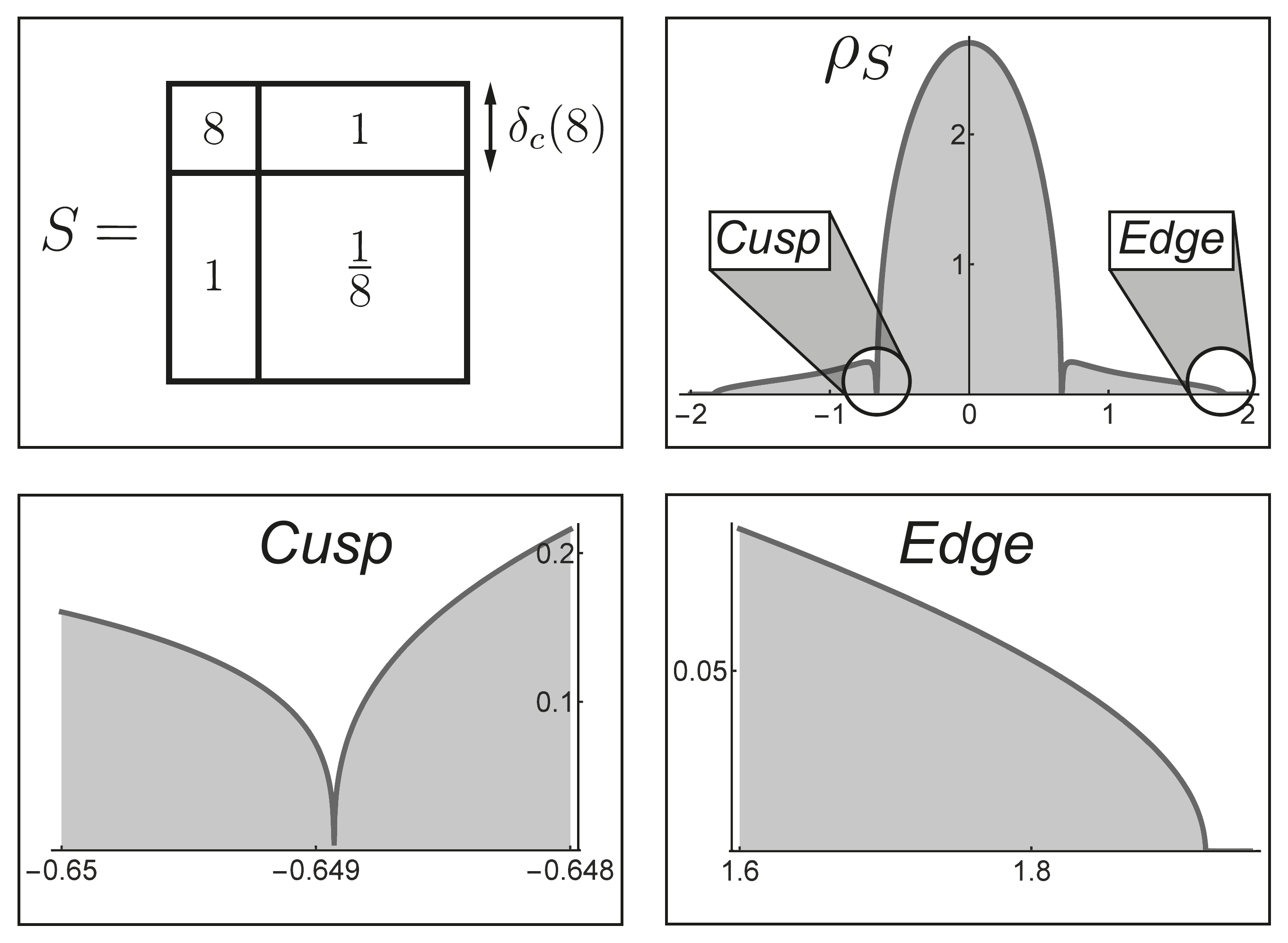}
	\caption{
	$2\times 2$ - block profile whose density has a cusp singularity. }
	\label{Fig:2by2Cusp}
\end{figure}

We remark that the solution $ m $ of the QVE corresponding to the $2\times 2$-profile \eqref{profile function} through $ s_{xy} = q(x,y) $ is of the form $ m(z) = \mu(z)\2 \mathbbm{1}_I + \nu(z)\2 \mathbbm{1}_{I^c}$, for some component functions $ \mu,\nu : \Cp \to \Cp $.
With this ansatz the original QVE reduces to a two dimensional system for the components $(\mu,\nu) $. 
Already in this simple case solving the QVE for general $\alpha,\beta,\gamma $ already requires solving a quartic polynomial. Nevertheless, Theorem \ref{thr:Regularity and singularities of the generating density} shows that a quartic singularity of $\rho_S$ never appears.

\subsection{Deformed Wigner matrices}

Another class of well studied self-adjoint random matrix models are the \emph{deformed Wigner matrices} of the form  
\[
H\,=\, A+\lambda V\,,
\]
where $A$ is a deterministic diagonal matrix and $V$ is a Wigner matrix, i.e., $V$ has centered i.i.d. entries, up to the symmetry constraints.
In the corresponding QVE \eqref{QVE}, 
we have $\xSpace=[0,1]$ with uniform measure, $s_{xy}\equiv \lambda$ and $a_x$ the smooth limiting profile of the diagonal entries of $A$. 
The average generating density $ \rho_{S,a} $ equals the asymptotic density of the eigenvalues as the dimension of $ H $ approaches infinity \cite{PasturDefWig}.
In particular, Theorem~\ref{thr:Regularity and singularities of the generating density} restricts the possible singularities of the limiting eigenvalue density to at most third order. 
 
The cubic root cusp has been observed in this context in \cite{BH} in the case when $V$ is a Gaussian unitarily invariant matrix (GUE) and $A$ has a spectrum symmetric to the origin with a gap around zero. As the coupling constant $\lambda>0$ increases from zero,  at a critical value the gap closes  in the  support of the density and a cubic singularity emerges. The argument in \cite{BH}, however,  did not use the QVE. Since the randomness is generated by a GUE matrix, all local correlation functions can be explicitly  computed via the Itzykson-Zuber formula. The cubic root cusp  can then be easily recovered from (3.22) in \cite{BH}. 
We remark that in this case the type of singularity also determines the local statistics.
Analogously to the Wigner-Dyson statistics in the bulk and the Tracy-Widom statistics at regular edges with a square root behavior, the cubic root singularity gives rise to a determinantal process described by the Pearcey kernel, see e.g. \cite{AM,BK,TW}.

\subsection{Translation invariant correlations}\label{sec:trinv}
Let $ \underline{\xi} = (\xi_{ij}:i,j\in \mathbb{Z}) $ be a family  of i.i.d. random variables indexed by $ \mathbb{Z}^2 $.
Let $ \theta_{pq} $, $ p,q \in \mathbb{Z} $, denote a shift of $ \underline{\xi} $, such that the $(i,j)$-th component of $ \theta_{p q}\underline{\xi} $ is $ \xi_{i+p,j+q} $. 
Given a measurable function  $ \Phi : \R^{\mathbb{Z}^2} \to \R $, such that 
\[
\mathbb{E}\,\Phi(\underline{\xi}) = 0\,,
\qquad
\mathbb{E}\,\Phi(\underline{\xi})^2 < \infty\,, 
\quad\text{and}\quad
\sum_{p,\1q \in \mathbb{Z}} \absb{\2\mathbb{E}\,\Phi(\theta_{pq}\underline{\xi})\1\Phi(\underline{\xi})\1} \,<\,\infty\,,
\]
we define a sequence of translation invariant random matrices 
 $H=H^{(N)}$
through 
\[
h_{ij} := \frac{\Phi(\theta_{ij}\underline{\xi})\1+\2\Phi(\theta_{ji}\underline{\xi})}{\!\sqrt{N\2}}
\,.
\]

Then their empirical spectral measures  converge weakly in probability to a non-random measure with a density $ \rho_S = \rho_{S\msp{-1},0}: \R \to [0,\infty) $. 
The limiting density $\rho_S$ is determined by the solution $ m $ of the QVE in the setup $ (\xSpace,\xMeasure(\dd x)) = ([0,1],\dd x)$, where   $ S $ has the integral kernel $ s : [\10,1\1]^2 \to [\10,\infty) $,
\[
s_{xy} \,:=\, 
4\sum_{p,\1q\in\mathbb{Z}}
\ee^{-\1\ii\12\1\pi\1(p\1x\2-\2q\1y)}
\,\mathbb{E}\,\Phi(\theta_{pq}\underline{\xi})\1\Phi(\underline{\xi}) 
\,,
\]
and $ a = 0 $.  
This convergence has been established in the Gaussian setup in \cite{AEK3,KhorunzhyPastur94,PasturShcerbinaAMSbook}.
In  \cite{BMP2013} it was extended by a comparison argument to the general setting we presented here.

\subsection{The color equation}\label{sec:color}

In \cite{AZdep} the authors show that the empirical distributions of the eigenvalues of a class of random matrices with dependent entries converge to a probability measure $ \mu $ on $ \R $ as the dimension of the matrices becomes large. The measure $ \mu $ is determined through the so-called \emph{color equations} (cf. equation (3.9) on p. 1135):
\[
\int_C \frac{s(c,c')P(\dd c')}{\lambda-\Psi(c',\lambda)}
\,=\,\Psi(c,\lambda)\,,
\qquad
\int_\R \frac{\mu(\dd \tau)}{\lambda-\tau} = \int_C\! \frac{P(\dd c)}{\lambda-\Psi(c,\lambda)}
\,,
\]
where $ \lambda \in \C $ and the color space is $ C= [0,1] \times \mathbb{S}^1 $, with $  \mathbb{S}^1 $  denoting the unit circle on the complex plane.
Identifying $ \xSpace = C $, $ x = c $,  $ y = c'$,  $ \xMeasure(\dd x) := P(\dd c) $, $ z = \lambda $, we see that the color equation is equivalent to the QVE \eqref{QVE}. Indeed, we have the correspondence 
\[
s_{xy} = s(c,c')\,,
\qquad
m_x(z) = \frac{\!-1\,}{\lambda-\Psi(c,\lambda)}
\,,
\]
so that from \eqref{m as stieltjes transform} we see that $ \rho_{S,0}(\dd \tau) = \mu(\dd \tau) $. 
Hence our results cover the asymptotic spectral statistics of this large class of random matrices with non-translation invariant correlations as well.

\section{Boundedness of the solution}

Recall the existence and uniqueness of the solution $ m $ to the QVE, as well as the Stieltjes transform representation (cf. Proposition \ref{prp:Existence and uniqueness}).
In this section we show that $ m $ is uniformly bounded, and that the imaginary part of $ m $ has mutually comparable components.
First we introduce a few notations and conventions that will be used throughout this paper.

\begin{notation}[Comparison relation]
For brevity we introduce the concept of \emph{comparison relations}: If $\varphi=\varphi(u) $ and  $ \psi=\psi(u) $  are non-negative functions on some set $ U $, then the notation $ \varphi \lesssim \psi $, or equivalently, $ \psi \gtrsim \varphi $, means that there exists a constant $ 0 < C < \infty $, depending only on the data input $s$ and $a$ of the QVE, such that $ \varphi(u) \leq C\1\psi(u) $ for all $ u \in U $. If $ \psi \lesssim \varphi \lesssim \psi $ then we write $ \varphi \sim \psi $, and say that $\varphi$ and $\psi$ are \emph{comparable}. Furthermore, we use $ \psi = \varphi + \ord(\xi) $ as a shorthand for $ \abs{\psi-\phi} \lesssim \xi $, where $\psi$ and $\varphi$ do not have to be non-negative.
\end{notation}

Many quantities in the following depend on the spectral parameter $z$, but for brevity we will often drop this dependence from our notation whenever the spectral parameter is considered fixed, e.g., we denote $ m = m(z) $.

Proposition \ref{prp:Existence and uniqueness} shows that the component $m_x$ of the solution to the QVE is determined by the component $v_x(\dd \tau)$ of the generating measure with support in $[-\kappa,\kappa]$. The Stieltjes transform representation \eqref{m as stieltjes transform} also implies that $v_x(\dd \tau)$ can be viewed as the weak limit $\lim_{\eta\downarrow 0}\frac{1}{\pi}\im m_x(\tau+\ii\1\eta)\dd\tau$. 
We may therefore restrict our analysis of $m$ to spectral parameters $z$ with $\abs{z}\leq 2\kappa$. 
The following Proposition is the main result of this section. 

\begin{proposition}[Uniform bound]
\label{prp:uniform bound}
Let $s$ and $a$ satisfy assumptions \textbf{(A)}, \textbf{(B)} and \textbf{(C)}.
Then the solution $m_x(z)$ of the QVE is uniformly bounded and bounded away from zero,
\begin{equation}
\label{uniform bound}
\abs{m(z)}\,\sim\, 1\,, \qquad \abs{z}\1\leq\1 2\kappa\,.
\end{equation} 
Moreover, the components of the imaginary part of $m$ are all comparable in size,
\begin{equation}
\label{comparable imaginary part}
\im m_x(z)\,\sim\,\im m_y(z)\,, \qquad  x,y \in \xSpace\,, \; \abs{z}\1\leq\1 2\kappa\,.
\end{equation}
\end{proposition}

In the following definition we introduce a  $z$-dependent operator $F(z)$ that depends on the value $m(z)$ of the solution at $z$. This operator will play a fundamental role in the upcoming analysis and in the proof of Proposition~\ref{prp:uniform bound} in particular. 

\begin{definition}[Operator $F$]
\label{def:operator F}
Let $m$ be the solution of the QVE and $z \in \Cp$. Then we define the operator $F(z)$ on $\Bounded(\xSpace,\C)$ by 
\begin{equation}
\label{definition of F}
(F(z)\1w)_x \,:=\, |\1m_x(z)| \int s_{x y} \1|\1m_y(z)|\1 w_y\1\xMeasure(\dd y)\,, \qquad w \in \Bounded(\xSpace,\C)\,.
\end{equation}
\end{definition}

We denote by $\norm{\2\cdot\2}_2$ the standard norm on $L^2(\xSpace)=L^2(\xSpace,\pi)$ and by $\norm{T}_2$ the induced operator norm of some operator $T$ on $L^2(\xSpace)$. We will now see that the diagonal rescaling by $\abs{m}$ in the definition of $F$ implies that the spectral radius of this operator will always be bounded by $1$. 

\begin{lemma}[Spectral radius of $F$]
\label{lmm:Spectral radius of F}
Let $z \in \Cp$ and the operator $F(z)$ be as in \eqref{definition of F}. Then the spectral radius $\norm{F(z)}_2$ is an eigenvalue of $F(z)$ and there exists at least one corresponding non-negative eigenfunction
$f(z) \in \Bounded(\xSpace,\R_0^+)$.  Any such eigenfunction satisfies the relation 
\begin{equation}
\label{relation between rho and alpha}
\norm{F(z)}_2\,=\, 1-\frac{\avg{f(z)\1|m(z)|}\im z}{\avg{f(z)|m(z)|^{-1}\im m(z)}}\,.
\end{equation}
\end{lemma}
\begin{proof}
Since the kernel $s$ is bounded and the solution of the QVE satisfies the trivial bound \eqref{m trivial bound}, the kernel of the integral operator $F$ is bounded as well. This implies that $F$ is compact when it is considered as an operator on $L^{p}(\xSpace)$ with $p \in (1,\infty)$. By the Krein-Rutman theorem the spectral radius $\norm{F}_2$ is an eigenvalue of $F$ with a corresponding non-negative eigenfunction $f \in L^{p}(\xSpace)$. From the eigenvalue equation, $Ff=\norm{F}_2\1f$, and the boundedness of the kernel of $F$ we read off that $f$ is bounded, i.e., up to modification on a set of measure zero it is an eigenfunction of $F$ as an operator on $\Bounded(\xSpace, \C)$. 

We will now show \eqref{relation between rho and alpha}. We take the imaginary part on both sides of \eqref{QVE} and multiply with $\abs{m}$:
\begin{equation}
\label{imaginary part of QVE}
|m|^{-1}\im m\,=\, |m|\1\im z + F ( |m|^{-1}\im m)\,.
\end{equation}
Let $f \in \Bounded(\xSpace, \R_0^+)$ be an arbitrary eigenfunction of $F$ corresponding to the eigenvalue $\norm{F}_2$. We multiply both sides of \eqref{imaginary part of QVE} with $f$ and take the average. Since the kernel of $F$ is symmetric we end up with
\[
\avg{\1f\1|m|^{-1}\im m\1}\,=\, \avg{\1f\1|m|\1}\1\im z \,+\2 \norm{F}_2\1 \avg{\1f\1 |m|^{-1}\im m\1}\,.
\]
This is equivalent to \eqref{relation between rho and alpha} and Lemma \ref{lmm:Spectral radius of F} is proven. 
\end{proof}

\begin{proof}[Proof of Proposition \ref{prp:uniform bound}] 
First we show the uniform upper bound in \eqref{uniform bound}. The uniform lower bound will be shown along the way as well. We fix $z \in \Cp$ with $\abs{z}\leq 2\1\kappa$.  We start by establishing boundedness in $L^2(\xSpace)$, i.e.,
\begin{equation}
\label{L2 bound}
\norm{m}_2\,\lesssim\, 1\,.
\end{equation}
Once we have shown \eqref{L2 bound}, the uniform bound follows from the following lemma, whose proof is postponed until the end of this section.
\begin{lemma}
\label{lmm:L2 bound implies uniform bound}
Let $ z \in \Cp $, $ \Phi \in \R^+ $ and $\Gamma_{\!\Phi}: \R^+ \to \R^+$ be the strictly monotonically increasing function defined by
\begin{equation}
\Gamma_{\!\Phi}(\tau) := \inf_{x \1\in\1 \xSpace}\int_\xSpace\!
\frac{\,\xMeasure(\dd y)}{(\2\tau^{-1}\!+|a_x-a_y|+\norm{S_x-S_y}_2\2\Phi\2)^{\12}\!} 
\;.
\end{equation}
If $ \norm{m(z)}_2 \leq \Phi $ and $ \lim_{\tau \to \infty}\Gamma_\Phi(\1\tau\1)>\Phi^2$, then 
\begin{equation}
\label{uniform bound on m}
\norm{\1m(z)} 
\;\leq\;
(\1\Gamma_{\Phi})^{-1}\msp{-1}(\Phi^2)
\,.
\end{equation}
\end{lemma} 
In fact, in our setting we have $\Gamma_{\!\Phi}(\R^+)=\R^+$ for all $\Phi$ because of the assumption \textbf{(C)} and therefore Lemma \ref{lmm:L2 bound implies uniform bound} always provides a uniform upper bound, given that the solution is bounded in $L^2(\xSpace)$.

The proof of \eqref{L2 bound} starts by establishing a bound in $L^1(\xSpace)$ first. For this we use that the quadratic form corresponding to $F$, evaluated on the constant function $1$, is bounded by $\norm{F}_2$. This spectral norm in turn  can be estimated by $\norm{F}_2\leq 1$, which is a consequence of  \eqref{relation between rho and alpha}. We therefore find the chain of inequalities,
\[
1 
\,\geq\,
\norm{F}_2
\,\geq\, 
\avg{\1\abs{m}\2S\1\abs{m}\1}
\,\geq\, 
\avg{(R\1|m|)^2}
\,\geq\, 
\avg{\1R\1|m|\1}^2
\,\geq\, 
\avg{|m|}^2 
\inf_{x \in \xSpace} (R1)^2_x
.
\]
In the third inequality we employed \textbf{(A)}, where $R$ has $r_{x y}$ as its kernel, and in the second to last step Jensen's inequality was used.  Thus, $\avg{|m|}\lesssim 1$. 

This $L^1(\xSpace)$-bound is now used to infer that $m$ is bounded away from zero. Indeed, taking absolute value on both sides of the QVE implies 
\[
\frac{1}{|m|}\,\leq\,|z| + \norm{a} + \avg{\1|m|\1}\!\sup_{x,y\1\in\1 \xSpace}s_{xy}
\,.
\]
In particular, we have $\abs{m}\gtrsim 1$ because $\abs{z}\leq 2\1\kappa$, which proves the uniform lower bound in \eqref{uniform bound}. 

From this lower bound on the absolute value of the solution, \eqref{L2 bound} follows by using that the $L^2(\xSpace)$-norm of $F$ applied to the constant function $1$ is bounded by the spectral radius of $F$,
\begin{equation}
\label{L2 bound from lower bound}
1
\,\geq\,
\norm{F}_2^2
\,\geq\,
\avg{(\abs{m} S \abs{m})^2}
\,\geq\, 
\avg{\1|m|^2}
\inf_{x \in \xSpace}(S|m|)_x^2
\,.
\end{equation}
The lower bound on $\abs{m}$ indeed yields
\[
\inf_{x \1\in\1 \xSpace}(S|m|)_x\,\geq\, \inf_{x \1\in\1 \xSpace}(R^2|m|)_x\,\gtrsim\inf_{x \1\in\1 \xSpace}\pB{\int r_{xy}\1 \xMeasure(\dd y)}^2 >\20
\,,
\]
where we used assumption \textbf{(A)} in the first and last inequality.
Thus, \eqref{L2 bound from lower bound} implies \eqref{L2 bound} and hence the uniform bound  \eqref{uniform bound}.

Now we show \eqref{comparable imaginary part}.  In fact, we will see that the components of $\im m$ are comparable to their average, $\im m_x\sim \avg{\1\im m\1}$. We take the imaginary part on both sides of \eqref{QVE} and use \eqref{uniform bound} to see that
\begin{equation}
\label{imaginary part of QVE as inequality}
\im m\,\sim\, \im z + S \im m\,.
\end{equation}
Since $\im z\geq 0$, we have $\im m\gtrsim S\im m$. We iterate this inequality $K$ times and find 
\[
\im m\,\gtrsim \, S^K\im m\,\gtrsim\, \avg{\1\im m\1}\,,
\]
using assumption \textbf{(B)}.

On the other hand, since $\im m \geq 0$, by averaging on both sides of \eqref{imaginary part of QVE as inequality} we get $\avg{\im m}\gtrsim \im z$. Therefore, \eqref{imaginary part of QVE as inequality} also implies
\[
\im m \,\lesssim\, \avg{\1\im m\1} + S \im m\,\leq\,  \avg{\1\im m\1} + \avg{\1\im m\1} \sup_{x,y\in \xSpace}s_{xy} \,.
\]
This finishes the proof of Proposition \ref{prp:uniform bound}.
\end{proof}
 
\begin{proof}[Proof of Lemma \ref{lmm:L2 bound implies uniform bound}]
By adding and subtracting $ m_x^{-1} $ and using the QVE we obtain
\[
\frac{1}{\abs{m_y}}
\,=\,
\absbb{\frac{1}{m_x\!} +a_x-a_y+\avgb{S_x-S_y,m}}
\,\leq\,
\absB{\frac{1}{m_x\!}} + |a_x-a_y| +\norm{S_x-S_y}_2\norm{m}_2
\,,
\]
for any $ y \in \xSpace $. Applying $ t \mapsto t^{-2} $ on both sides, integrating over $ y $, and using $ \norm{m}_2 \leq \Phi $, we find   
\[
\Phi^2 
\ge\! \int_\xSpace \absB{\frac{1}{\2m_y}\!}^{-2}\msp{-7}\xMeasure(\dd y)
\ge \!\int_\xSpace \Big(\frac{1}{\abs{\1m_x}\!} + |a_x-a_y| + \norm{S_x-S_y}_2\2\Phi\Big)^{\!-2}\msp{-5}
\xMeasure(\dd y)
\,\ge\,
\Gamma_{\Phi}\big(\abs{m_x}\big)
\,,
\]
for any $ x \in \xSpace $.
This is equivalent to \eqref{uniform bound on m} since the strictly monotonous function  $\Gamma_{\!\Phi} $ is invertible on $\Phi^2$ by assumption.
\end{proof}

\section{Regularity of the solution}
\label{sec:Regularity of the solution}

From here on we will always assume that $s$ and $a$ satisfy assumptions \textbf{(A)}, \textbf{(B)} and \textbf{(C)}. 
In this section we will analyze the regularity of $\im m(z)$ as a function of $z$. Since the generating density is the limit of $\frac{1}{\pi}\im m(\tau+\ii \eta)$ as $\eta\downarrow 0$, this regularity will be inherited by $v$, provided it is established uniformly in $\eta$. In this spirit, we prove the following proposition.

\begin{proposition}[Regularity of the generating density]
\label{prp:regularity}
The generating density $v(\tau) \in \Bounded(\xSpace,\R_0^+)$ exists and is uniformly $1/3$-H\"older continuous in $\tau$,
\[
\sup_{\tau_1 \neq \tau_2}\frac{\norm{\1v(\tau_1)-v(\tau_2)}}{|\tau_1-\tau_2|^{1/3}}\,<\, \infty\,.
\]
Moreover, $v(\tau)$ is real analytic for $\tau \in \R\backslash \partial \mathfrak{S} $ .
\end{proposition}

From the comparability \eqref{comparable imaginary part} of the components of $\im m$ we infer that the components $v_x$ of the generating density are comparable as well. In particular, the set $\mathfrak{S}$ as defined in \eqref{definition of set S} does not depend on $x$

\begin{proof}[Proof of Proposition \ref{prp:regularity}]
In order to show that the generating density exists and is $1/3$-H\"older continuous, we will prove that 
\begin{equation}
\label{Holder continuity of im m}
\sup_{z_1 \neq z_2}\frac{\norm{\1\im m(z_1)-\im m(z_2)}}{|z_1-z_2|^{1/3}}\,< \, \infty\,,
\end{equation}
where the supremum is taken over $z_1,z_2 \in \Cp$. In particular, the imaginary part of $m$ can be extended to the real line as a $1/3$-H\"older continuous function on the extended upper half plane $\im m: \overline{\Cp} \to \Bounded(\xSpace, \R^+_0)$.  
Due to the Stieltjes transform representation \eqref{m as stieltjes transform}, this extension  coincides with the generating density on the real line up to a factor of $\pi$, i.e.,
\[
v(\tau)\,=\, \pi^{-1}\im m(\tau)\,,\qquad \tau \in \R\,.
\]

If the supremum in \eqref{Holder continuity of im m} is taken over $z_1$  away from the support of the generating measure, $\abs{z_1}\geq 2\1\kappa$, then the finiteness  follows from \eqref{m as stieltjes transform}, the boundedness of $\im m$ (cf. \eqref{uniform bound}) and $\mathrm{supp} \2v_x \subseteq [-\kappa, \kappa\1]$. By symmetry, the same argument is made for $\abs{z_2}\geq 2\1\kappa$.

To prove \eqref{Holder continuity of im m} with the supremum taken over $z_1,z_2 \in \Cp$ with $\abs{z_1},\abs{z_2}\leq 2\kappa$ we differentiate both sides of \eqref{QVE} with respect to $z$, multiply by $m^2$ and collect the terms involving $\partial_z m$ on the left hand side,
\begin{equation}
\label{derivative of QVE}
( 1-m(z)^2S\1) \partial_z m(z) \,=\, m(z)^2.
\end{equation}
By the following lemma we can invert the operator $ 1-m(z)^2S$.
\begin{lemma}[Bulk stability]
\label{lmm:Bounds on B-inverse}
Let $z \in \Cp$ with $\abs{z}\leq2\kappa$. Then 
\begin{equation}
\norm{(\11-m(z)^2S\1)^{-1}} 
\,\lesssim\,
\avg{\2\im m(z)}^{-2}.
\end{equation}
\end{lemma}
The proof of Lemma \ref{lmm:Bounds on B-inverse} is provided at the end of this section. 
We apply the lemma to \eqref{derivative of QVE} and find $\norm{\partial_z m(z)}\lesssim  \avg{\im m(z)}^{-2}$ for  spectral parameters $\abs{z}\leq 2\1\kappa$. Since $z \mapsto m(z)$ is analytic the Cauchy-Riemann equations yield 
$ 2\1\ii\partial_z \im m = \partial_zm $. 
We infer that
\begin{equation}
\label{differential Holder inequality}
\abs{\partial_z \im m\1}\,\lesssim\, \avg{\1\im m\1}^{-2}\,\lesssim\, (\im m)^{-2},
\end{equation}
where we used \eqref{comparable imaginary part} in the last inequality. The differential inequality \eqref{differential Holder inequality} directly implies \eqref{Holder continuity of im m}.

It remains to show that $v$ is analytic on $\mathfrak{S}$. 
We fix $ \tau_0  $, with $ \avg{v(\tau_0)}> 0 $, and consider the complex ODE,
\begin{equation}
\begin{split}
\label{ODE for m}
 \partial_zw \,&=\,(1-w^2S)^{-1}w^2\,,
\quad w(\tau_0) \2=\2 w_0\,,
\end{split}
\end{equation}
for an analytic function $w: \mathbb{D}_\eps(\tau_0) \to \Bounded(\xSpace,\C)$ on the disc of radius $\eps>0$ around $\tau_0$. 
As initial value we choose $w_0:=m(\tau_0)$. The right hand side of the ODE is an analytic function on a neighborhood of $w_0$ in $\Bounded(\xSpace,\C)$ because $\norm{(1-w_0^2S)^{-1}}\lesssim \avg{v(\tau_0)}^{-2}$ initially by Lemma \ref{lmm:Bounds on B-inverse}.
By standard methods the ODE locally has a unique analytic solution that coincides with the solution $m$ of the QVE on $\mathbb{D}_\eps(\tau_0) \cap \overline{\Cp}$, provided $\eps$ is sufficiently small.
\end{proof}

\begin{proof}[Proof of Corollary \ref{crl:Holder-regularity of the solution}]
The Stieltjes transform of a H\"older continuous function is again H\"older continuous with the same exponent. This is formally expressed by the following lemma. We refer to, e.g.~\cite{Mu}, Sec.~22, for details.
\end{proof}

\begin{lemma}
\label{Holder continuity}
Let $\mu$ be a uniformly $\alpha$-H\"older continuous function on $\R$ with $\alpha \in (0,1)$. Then its Stietjes-transform,
\[
(M\mu)
(z)\,=\, \int_\R \frac{\mu(\tau)\dd \tau}{\tau\2-\2z}\,,\qquad z \in \Cp\,,
\]
is uniformly $\alpha$-H\"older continuous on $\Cp$. In particular, $M\mu$ can be extended to a uniformly $\alpha$-H\"older continuous function on $\overline{\Cp}$.
\end{lemma}

For the proof of Lemma \ref{lmm:Bounds on B-inverse} we will need more spectral information about the operator $F$ than the simple formula \eqref{relation between rho and alpha} for its spectral radius. In particular, we need a uniform spectral gap, whose formal definition is as follows.

\begin{definition}[Spectral gap]
\label{def:Spectral gap}
Let $ T:L^{2}(\xSpace) \to L^{2}(\xSpace) $ be a compact self-adjoint operator. The spectral gap $ \mathrm{Gap}(T)\geq 0 $ is the difference between the two largest eigenvalues of $ \abs{T} $ (defined by spectral calculus). If $ \norm{T}_{2}$ is a degenerate eigenvalue of $ \abs{T} $, then $ \mathrm{Gap}(T) = 0 $. 
\end{definition}

\begin{lemma}[Spectral gap of $F\2$]
\label{lmm:operator F}
Let $F(z)$ be as in \eqref{definition of F} for $\abs{z}\leq 2 \kappa$. Then the spectral radius $\norm{F(z)}_2\sim 1$
 is a non-degenerate eigenvalue
with corresponding $\norm{\2\cdot\2}_2$-normalized non-negative eigenfunction $ f(z) \in \Bounded(\xSpace,\R_0^+)$ satisfying
\begin{equation}
\label{f comparable to 1}
 f(z) \,\sim\, 1\,.
\end{equation}
The operator $ F(z)$ has a uniform spectral gap
\begin{equation}
\label{F spectral gap}
\mathrm{Gap}( F(z))\,\sim\, 1\,.
\end{equation}
\end{lemma}
\begin{proof}
Since $\abs{m}\sim 1$ the operator $F$ has the property \textbf{(B)}  in place of $S$. Therefore, $T:=F^K/\norm{F^K}_2$ has a symmetric non-negative kernel $t_{x y}\sim 1$. In particular, $T$ is compact, when viewed as an operator on $L^{2}(\xSpace)$, and by the Krein-Rutman theorem its spectral radius $\norm{T}_2=1$ is a non-degenerate eigenvalue with corresponding normalized non-negative eigenfunction $h \in L^{2}(\xSpace)$. By the pointwise boundedness of the kernel of $T$ from both above and below, the eigenvalue equation $h=Th$ implies that $h\sim 1$. 
The result follows from Lemma \ref{lmm:Spectral gap for operators with positive kernel} below, noticing that $f=h$.
\end{proof}

\begin{lemma}[Spectral gap for operators with positive kernel]
\label{lmm:Spectral gap for operators with positive kernel}
Let $T$ be a symmetric compact integral operator on $L^{2}(\xSpace) $ with a non-negative integral kernel $ t_{xy} = t_{yx} \ge 0 $. 
Then
\[
\mathrm{Gap}(T) \,\ge\, 
\biggl(\!\frac{\,\norm{h}_2\!}{\norm{h}}\!\biggr)^2\!\inf_{x,y \in \xSpace}t_{x y}
\,,
\]
where $ h $ is an eigenfunction with $ Th = \norm{T}_2\1h$. 
\end{lemma}
For the convenience of the reader we provide the proof of this standard result in the appendix. 

\begin{definition}[Eigenfunction $f\2$]
Let $z \in \Cp$ and $F(z)$ be as in \eqref{definition of F}. Then by Lemma \ref{lmm:operator F} the eigenvalue $\norm{F(z)}_2$ is non-degenerate. We will always denote by $f(z)$ the corresponding $\norm{\2\cdot\2}_2$-normalized non-negative eigenfunction.
\end{definition}

\begin{proof}[Proof of Lemma \ref{lmm:Bounds on B-inverse}] 
We fix $z \in \Cp$ with $\abs{z}\leq2\kappa$ and introduce the notation
\[
B\,:=\, m^{-2}\abs{m}^2- F\,.
\]
First we notice that it suffices to estimate the norm of $B^{-1}$ as an operator on $L^2(\xSpace)$ from above, because
\begin{equation}
\label{infinity to infinity from L2 to L2}
\norm{\1(\11-m^2S\1)^{-1}} 
\,\lesssim\,
\norm{B^{-1}}\,\lesssim\,1+ \norm{B^{-1}}_2 \,.
\end{equation}
Indeed, the first inequality follows from \eqref{uniform bound} and 
\[
(\11-m^2S\1)w \,=\, \abs{m}^{-1}m^2 B (\abs{m}^{-1}w)\,,\qquad w \in \Bounded(\xSpace,\C)
\,.
\]
For the second inequality in \eqref{infinity to infinity from L2 to L2} we use the following argument. 

Suppose $B^{-1}$ exists and is bounded on $L^2(\xSpace)$ and let $h \in  \Bounded(\xSpace, \C)$. Then there is some $w \in L^2(\xSpace)$ such that $Bw=h$. In particular, by the definition of $B$ we have
\[
\norm{w}\,
\leq\,
\norm{h}+\norm{Fw} 
\,\leq\, 
\norm{h}+ \norm{m}^2\avg{\abs{w}}\sup_{x,\1y \1\in\1 \xSpace}s_{x y}
\,\lesssim\,
\norm{h}+\norm{w}_2
\,=\, \norm{h}+\norm{B^{-1}h}_2
\,.
\]
Now we use the boundedness of $B^{-1}$ on $L^2(\xSpace)$, i.e., $\norm{B^{-1}h}_2\,\leq\, \norm{B^{-1}}_2\norm{h}$, and \eqref{infinity to infinity from L2 to L2} follows. 

It remains to show that $\norm{B^{-1}}_2\lesssim \avg{\1\im m\1}^{-2}$. We apply the following lemma, which is proven in the appendix.
\begin{lemma}
\label{lmm:Norm of B^-1-type operators on L2}
Let $ T $ be a compact self-adjoint and $U$ a unitary operator on $ L^{2}(\xSpace) $. 
Suppose that $ \mathrm{Gap}(T)> 0 $ and $\norm{T}_2\leq 1 $. 
Then there exists a universal positive constant $C$ such that
\begin{equation}
\label{bound on U-T inverse}
\norm{(\1U-T\1)^{-1}}_2
\,\leq\, 
\frac{C}{ \mathrm{Gap}(T)\,\abs{\11 -\norm{T}_2\langle\1 h, Uh\rangle\1}}
\,,
\end{equation}
where $ h $ is the normalized eigenvector, corresponding to the non-degenerate eigenvalue 
$\norm{T}_2$ of  $ T$ and $\scalar{\2\cdot\2}{\2\cdot\2}$ denotes the standard scalar product on $L^2(\xSpace)$.
\end{lemma}
With the choice $T:=F$, the unitary multiplication operator $U:=m^{-2}\abs{m}^2$  and $\norm{F}_2\leq 1$ by \eqref{relation between rho and alpha}, the lemma implies
\begin{equation}
\norm{B^{-1}} _2
\,\leq\, 
\frac{C}{\mathrm{Gap}(F)\2\abs{\21 - \norm{\1F\1}_2\avg{\1m^{-2}\abs{m}^2f^{\12}\1}}\2}
\,\lesssim\, 
\frac{1}{
\re\bigl[\11 -\avg{\1\abs{m}^{-2}m^2f^{\12}\1}\1\bigr]
}
\,,
\end{equation}
where we used $\mathrm{Gap}(F) \sim 1$ from Lemma \ref{lmm:operator F} for the second inequality. Since we also have \eqref{f comparable to 1}, \eqref{uniform bound} and \eqref{comparable imaginary part} at our disposal, we can estimate further,
\begin{equation}
\re \bigl[\2
1 -\avg{\1\abs{m}^{-2}m^2f^{\12}}
\bigr]
\,=\,2\1\avgb{\1\abs{m}^{-2}(\1\im m)^2f^{\12}\1}\,\gtrsim \,\avg{\1\im m\1}^2.
\end{equation}
\end{proof}

\section{Singularities}
\label{sec:Singularities}

In this section we will prove our main result, Theorem \ref{thr:Regularity and singularities of the generating density}. Its proof relies on a careful analysis of how $m(z)$ changes in $z$ in a neighborhood of a singular point $\tau \in \partial \mathfrak{S}$. This change will be described in leading order by a complex valued scalar function $\Theta$ which satisfies a cubic equation. As a solution to this cubic equation, $\Theta$ can only give rise to square root or cubic root singularities of the generating density at $\tau$.

\begin{definition}[Extensions to the real line]
Under the assumptions \textbf{(A)}, \textbf{(B)} and \textbf{(C)} (which are always assumed here)
we extend the solution $m$ of the QVE and all its derived quantities, such as $F$, $f$, etc., to $\overline{\Cp}$ according to Corollary \ref{crl:Holder-regularity of the solution}. We denote these extensions with the same symbols. 
\end{definition}
 
\begin{proposition}[Cubic equation]
\label{prp:Cubic equation} 
 For any given $\tau \in \partial \mathfrak{S}$, the complex valued function 
\begin{equation}
\label{definition of Theta}
\Theta(\omega;\tau)\,:=\, \avgB{f(\tau) \,\frac{m(\tau+\omega)-m(\tau)}{|\1m(\tau)|}}\,,\qquad \omega \in [-\kappa,\kappa]\,,
\end{equation}
describes the change of $m$ around $\tau$ to leading order,
\begin{equation}
\label{leading order of change}
m(\tau+\omega)-m(\tau)
\;=\;
\Theta(\omega;\tau)\2|\1m(\tau)|\2f(\tau)+\ord\pb{ \abs{\Theta}^2+|\omega|}.
\end{equation}
 The function $\Theta$ solves the approximate cubic equation,
\begin{equation}
\label{cubic equation}
\psi(\tau)\2\Theta(\omega;\tau)^3+\sigma(\tau)\1\Theta(\omega;\tau)^2 +\avg{|m(\tau)|f(\tau)}\omega
\;=\: e(\omega;\tau) 
\,,
\end{equation}
where the error term $e(\omega;\tau)$ satisfies
\begin{subequations}
\label{bounds on e}
\begin{align}
\label{bound on |e|}
\abs{\1e(\omega;\tau)}\,&\lesssim\, |\omega|\abs{\Theta(\omega;\tau)}+|\omega|^2
,
\\
\label{bound on |Im e|}
\abs{\1\im e(\omega;\tau)}\,&\lesssim\, |\omega|\im \Theta(\omega;\tau)\,.
\end{align}
\end{subequations}
The real valued coefficients $\sigma$ and $\psi$ are defined as
\begin{equation}
\label{definition of sigma and psi}
\sigma(\tau) \,:=\, \avgb{ \1 f(\tau)^3\1 \mathrm{sign} \2 m(\tau) \1}\,,\qquad
\psi(\tau)\,:=\, \mathcal{D}_\tau\p{f(\tau)^2\mathrm{sign} \2m(\tau)} \,,
\end{equation}
where the non-negative quadratic form $\mathcal{D}_\tau$ is given by
\begin{equation}
\label{definition of D}
\mathcal{D}_\tau(w)\,:=\,
\avgB{w\2Q(\tau) \frac{1+F(\tau)}{1-F(\tau)}\1Q(\tau) w}\,,
\qquad Q(\tau) w\,:=\, w- \avg{f(\tau)w}f(\tau)\,,
\end{equation}
for any  $w \in \Bounded(\xSpace,\C)$.
The cubic equation for $\Theta$ is stable in the sense that
\begin{equation}
\label{stability of cubic}
\psi(\tau) + \sigma(\tau)^2 \,\sim\,1 \,, \qquad \tau \in [-\kappa,\kappa\2]\,.
\end{equation}
\end{proposition}

\begin{proof}[Proof of Proposition \ref{prp:Cubic equation}]
We fix $\tau \in \partial \mathfrak{S}$. In particular, $\im m(\tau)=0$. We will often neglect the dependence of various quantities on $\tau$ in our notation, e.g.  $m=m(\tau)$, $Q=Q(\tau)$, $\Theta(\omega)=\Theta(\omega;\tau)$, etc. We start the proof by showing that
\begin{equation}
\label{spectral radius 1 at singularity}
\norm{F(\tau)}_2\,=\, 1\,.
\end{equation}
Since $\tau \mapsto \norm{F(\tau)}_2$ is a continuous function, it suffices to show $\norm{F(\widetilde \tau)}_2=1$ for $\widetilde \tau \in \mathfrak{S}$. The relation \eqref{relation between rho and alpha} extends to $\widetilde \tau$ since the denominator on the right hand side is positive when evaluated at that point. In particular, the right hand side of this relation equals $1$ at $\widetilde \tau$ since $\im m(\widetilde \tau)>0$. 

We introduce the scaled difference between the solution of the QVE evaluated at $\tau+\omega$ and at $\tau$,
\[
u(\omega)\,:=\, \frac{m(\tau+\omega)-m(\tau)}{\abs{m(\tau)}}\,.
\]
Using that $m(\tau)=( \mathrm{sign}\2 m(\tau))\abs{m(\tau)}$, the definition of the operator $F$ from \eqref{definition of F} and the QVE with spectral parameters $z=\tau$ and $z=\tau+\omega$, it is easy to verify that $u$ satisfies the quadratic equation 
\begin{equation}
\label{quadratic for u}
(1-F)u(\omega) \,=\, p\1u(\omega)Fu(\omega) +\omega \abs{m} + \omega p\1\abs{m}u(\omega)\,,\quad p\,:=\, \mathrm{sign}\2 m
\,.
\end{equation}

We treat the direction $f$, which constitutes the kernel of $1-F$, separately. Recall that $Q=1-\avg{f \,\cdot\,}f$ is the projection onto the orthogonal complement of $f$. We decompose $u$ according to
\begin{equation}
\label{u decomposition}
u(\omega)\,=\,\Theta(\omega)\1 f + Q u(\omega)\,,\qquad \Theta(\omega)\,=\, \avg{f\1u(\omega)}\,.
\end{equation}

The H\"older continuity of $m$ from Corollary \ref{crl:Holder-regularity of the solution} leads to the a priori estimate
\begin{equation}
\label{a priori Holder estimate}
\abs{\Theta(\omega)}+\norm{Qu(\omega)}\,\lesssim\, \abs{\omega}^{1/3}.
\end{equation}
We will now derive an improved bound for $Qu$. 
To this end we insert the decomposition \eqref{u decomposition} into \eqref{quadratic for u}, use the eigenvalue equation $Ff=f$ and project with $Q$ on both sides. A short calculation shows that 
\begin{equation}
\label{Qu equation}
(1-F)Qu(\omega)
 \,=\,\Theta(\omega)^2\1Q[p\1f^2] +e_1(\omega)
\,,
\end{equation}
where the error function $e_{1}(\omega) \in \Bounded(\xSpace,\C)$ satisfies the two bounds
\begin{subequations}
\begin{align}
\label{e1 bound1}
\norm{e_1(\omega)}\,
&\lesssim\,
\abs{\omega}^{1/3}\norm{Qu(\omega)} + \abs{\omega} 
\,,
\\
\label{e1 bound2}
\norm{\im e_{1}(\omega)}\,
&\lesssim\,
\abs{\omega}^{1/3}\norm{\im Qu(\omega)}+ (\1\norm{Qu(\omega)}+\abs{\omega}\1)\2 \im \Theta(\omega)\,.
\end{align}
\end{subequations}
Here we have used \eqref{a priori Holder estimate}. 

Inverting $1-F$ on the orthogonal complement of $f$ in \eqref{Qu equation} and using \eqref{e1 bound1} and \eqref{e1 bound2}, respectively, yields the improved bounds
\begin{subequations}
\begin{align}
\label{Qu bound1}
\norm{Qu(\omega)}\,&\lesssim\, \abs{\Theta(\omega)}^2+ \abs{\omega}\,,
\\
\label{Qu bound2}
\norm{\im Qu(\omega)}\,&\lesssim\, (\1\abs{\Theta(\omega)}+\abs{\omega}\1)\1\im \Theta(\omega) \,.
\end{align}
\end{subequations}
For both inequalities, \eqref{e1 bound1} and \eqref{e1 bound2}, we have used that $1-F$ is invertible on its image with bounded inverse,
\begin{equation}
\label{1-F bounded inverse}
\norm{(1-F)^{-1}Q\1}\,\lesssim\; 1+\norm{(1-F)^{-1}Q\1}_2
\,\lesssim\; 1\,.
\end{equation}
The first inequality in \eqref{1-F bounded inverse} follows from the same argument as the second inequality in \eqref{infinity to infinity from L2 to L2}, while the second inequality is a consequence of the uniform spectral gap estimate  \eqref{F spectral gap} for $F$. Additionally, for \eqref{Qu bound1} we have used \eqref{e1 bound1} and for \eqref{Qu bound2} we have used \eqref{e1 bound2} in conjunction with \eqref{Qu bound1}. In particular, the improved bound \eqref{Qu bound1} on the norm of $Qu$ together with \eqref{u decomposition} shows the validity of \eqref{leading order of change}.

Now we will derive the cubic equation \eqref{cubic equation} for $\Theta$. We start by plugging the decomposition \eqref{u decomposition} into \eqref{quadratic for u}, using $Ff=f$ and projecting on both sides with the linear functional $w \mapsto \avg{f\1w}$ onto the $f$-direction,
\begin{equation}
\label{projection onto f}
0 
 \,=\, \Theta(\omega)^2 \avg{pf^3} + \Theta(\omega) \avg{p\1 f^2 (1+F)Qu(\omega)} + \omega \avg{f\abs{m}} + e_2(\omega)
\,.
\end{equation}
Here, the error term $e_{2}$ satisfies 
\begin{subequations}
\begin{align}
\label{e2 bound1}
\abs{e_2(\omega)}\,&\lesssim\, \norm{Qu(\omega)}^2 +\abs{\omega} \abs{\Theta(\omega)} + \abs{\omega} \norm{Qu(\omega)}\,,
\\
\label{e2 bound2}
\abs{\im e_2(\omega)}\,&\lesssim\,  \norm{Qu(\omega)} \norm{\im Qu(\omega)} +\abs{\omega} \abs{\im \Theta(\omega)} + \abs{\omega}\norm{\im Qu(\omega)}\,.
\end{align}
\end{subequations}
Solving for $ Qu $ in \eqref{Qu equation} and plugging the resulting expression into \eqref{projection onto f}  yields \eqref{cubic equation} with the coefficients $\sigma$ and $\psi$ defined as in \eqref{definition of sigma and psi} and the error term 
$e(\omega)=e_\tau(\omega)$ given by
\begin{equation}
\label{definition of e}
e(\omega)\,:=\, -\1e_2(\omega)\,-\, \Theta(\omega) \avgb{p\1 f^2 (1+F)(1-F)^{-1}e_1(\omega)} \,.
\end{equation}

It remains to verify the error bounds \eqref{bounds on e} and show the stability of the cubic \eqref{stability of cubic}. We start with the error bounds by estimating the absolute value,
\[
\abs{e(\omega)}\,\lesssim\, \abs{e_2(\omega)} + \abs{\Theta(\omega)}\norm{e_1(\omega)}\,\lesssim\,\abs{\omega}
\abs{\Theta(\omega)}+ \abs{\omega}^2,
\]
where in the second inequality we used \eqref{e2 bound1}, \eqref{e1 bound1}, \eqref{Qu bound1} and \eqref{a priori Holder estimate} in that order. 

Now we estimate the imaginary part of $e(\omega)$. From its definition \eqref{definition of e} we read off the first inequality in 
\[
\abs{\1\im e(\omega)}\,\lesssim\, \abs{\1\im e_2(\omega)} + \norm{e_1(\omega)}\2\im \Theta(\omega) + \abs{\Theta(\omega)}\norm{\1\im e_1(\omega)}
\,\lesssim\, 
\abs{\omega} \im \Theta(\omega)
\,.
\]
For the second estimate we used \eqref{e2 bound2}, \eqref{e1 bound1}, \eqref{e1 bound2}, \eqref{Qu bound1}, \eqref{Qu bound2} and \eqref{a priori Holder estimate} one after the other.

Finally, we show  \eqref{stability of cubic}. Only the lower bound requires a proof. First we observe that, by Definition \ref{def:Spectral gap} of the spectral gap, and by $\norm{F}_2 \leq 1$, the quadratic form $\mathcal{D}$, defined in \eqref{definition of D}, satisfies the lower bound
\[
\mathcal{D}(w)\,\geq\,\frac{ \mathrm{Gap}(F)}{2}\,\norm{Qw}_2^2\,,\qquad w \in \Bounded(\xSpace,\C)\,.
\]
With the choice $w:=pf^2$ we conclude that
\[
\psi+\sigma^2\,=\, \mathcal{D}(w)+\avg{f\1w}^2
\,\geq\, \frac{\mathrm{Gap}(F)}{2}\norm{Qw}_2^2 + \norm{(1-Q)w}_2^2\,\gtrsim\, \norm{w}^2_2\,\geq\, 1\,,
\]
where we used \eqref{F spectral gap} in the second to last estimate and the normalization of $f$ together with Jensen's inequality in the last step.
This finishes the proof of Proposition \ref{prp:Cubic equation}.
\end{proof}

We are now ready to prove Theorem \ref{thr:Regularity and singularities of the generating density}. Proposition \ref{prp:Cubic equation} reveals that the difference $m(\tau+\omega)-m(\tau)$ at a singular point $\tau \in \partial \mathfrak{S}$ is mainly determined by $\Theta(\omega)$. The cubic equation \eqref{cubic equation} for this quantity is stable in the sense that the second and third order coefficient cannot vanish at the same time (cf. \eqref{stability of cubic}). We therefore expect $\Theta$ to only allow for algebraic singularities of order not higher than three. This expectation is supported by the $1/3$-H\"older regularity, established in Corollary \ref{crl:Holder-regularity of the solution}. In fact, since all solutions of \eqref{stability of cubic} can be found to leading order explicitly, most of the proof of Theorem \ref{thr:Regularity and singularities of the generating density} is concerned with selecting the correct solution branch of \eqref{cubic equation}. 

\begin{proof}[Proof of Theorem \ref{thr:Regularity and singularities of the generating density}]
Taking into account the statements of Propositions \ref{prp:uniform bound} and \ref{prp:regularity}, it remains to show the behavior \eqref{cusp singularity} and \eqref{edge singularity} as the value of the generating density approaches zero and  that $\mathfrak{S}$ consists of only finitely many intervals. The latter will be shown at the very end of the proof.

Let us fix $\tau \in \partial \mathfrak{S}$. We start by considering the case where $\sigma(\tau)=0$, with $\sigma$ given by \eqref{definition of sigma and psi}.  Within the proof we will see that this characterizes the cusp singularities.

\medskip
\noindent{\scshape Cusp:} Let $\sigma(\tau)=0$. We apply Proposition \ref{prp:Cubic equation}. The cubic equation \eqref{cubic equation} takes the simplified form
\begin{equation}
\label{cusp equation}
\psi\2\Theta(\omega)^3+\avg{|m|f\1}\1\omega
\,=\,
e(\omega)\,,
\end{equation}
where $e(\omega)$ satisfies \eqref{bounds on e} and $\psi \sim 1$ according to \eqref{stability of cubic}. 
Since $ \abs{m},f \sim 1  $ and $ m(z) $ is uniformly $ 1/3$-H\"older continuous in $ z $ (cf. \eqref{Holder continuity of m}), the function $ \Theta(\omega) $ inherits the regularity of $ m $ by its definition in \eqref{definition of Theta}.
 In particular, $\abs{\Theta(\omega)}\lesssim \abs{\omega}^{1/3}$ and \eqref{bound on |e|} implies
\[
e(\omega)\,=\, \ord(\abs{\omega}^{4/3})\,.
\]

A simple perturbative calculation of the solution of \eqref{cusp equation} shows that $\Theta$ has the form 
\begin{equation}
\label{cusp solution for Theta}
\Theta(\omega)
\;=\;
\widehat \Theta_{p}(\omega)\mathbbm{1}(\omega < 0)\2+\,\widehat\Theta_{q}(\omega)\mathbbm{1}(\omega \geq 0)\,+\,
\ord(\abs{\omega}^{2/3})\,,
\end{equation}
where $\widehat \Theta_\#$ is the solution of \eqref{cusp equation} if the error term on the right hand side is set equal to zero, i.e., 
\begin{equation}
\label{ideal cusp solution}
\widehat \Theta_\#(\omega)\,:=\, -\2\zeta_\# \;  (\mathrm{sign}\, \omega)\pbb{\!\frac{\avg{|m|f}}{\psi}\,\abs{\omega}}^{\!1/3},\qquad \#=0,\pm\,.
\end{equation}
Here, $\zeta_0=1$ and $\zeta_\pm=(-1\pm\ii \sqrt{3})/2$ are the cubic roots of unity. We will now show that $p=+$ and $q=-$ for the indices in \eqref{cusp solution for Theta}. 

 First we rule out the choice $p=-$, as well as the choice $q=+$. Indeed, in both cases the imaginary part of the corresponding $\Theta(\omega)$ would have negative values for $\omega < 0$ and $\omega >0$, respectively. This is contradictory to the definition of $\Theta$ in \eqref{definition of Theta} and to the fact that $\im m(\tau)=0$ and $\im m(\tau+\omega)\geq 0$.

Suppose now that $p=0$ or $q=0$. Then from the explicit form \eqref{ideal cusp solution} of the leading order term to $\Theta$ we read off that there is a positive constant $c_1\sim 1$ such that
\begin{equation}
\label{cusp bounds on Theta}
\abs{\1\re \Theta(\omega)} \,\sim\, \abs{\omega}^{1/3},\qquad \im \Theta(\omega)\,\lesssim\, \abs{\omega}^{2/3},
\end{equation}
on the corresponding side of $\omega=0$, i.e., for $\omega \in [-c_1,0\1]$ or $\omega \in  [\10,c_1]$, respectively. Taking the imaginary part on both sides of the cubic \eqref{cusp equation} this implies that
\begin{equation}
\label{imaginary part gap at cusp}
\abs{\omega}^{2/3}\im \Theta(\omega)\,\lesssim\, (\im \Theta(\omega))^3+\abs{\omega}\im \Theta(\omega)
\,\lesssim\, \abs{\omega}\im \Theta(\omega)
\,,
\end{equation}
where we used the estimate \eqref{bound on |Im e|} on the error term $e(\omega)$ to get the first estimate.
For the second inequality we have  used \eqref{cusp bounds on Theta} to bound $\im \Theta$.
From \eqref{imaginary part gap at cusp} we conclude that there is a positive constant $c_2\sim 1$ such that
\begin{equation}
\im \Theta(\omega)\,=\, 0 \qquad \text{ for }\quad 
\begin{cases}
\omega \in [-c_2,0\1] & \text{ if } \; p\2=\20\,, 
\\
\omega \in [\10,c_2] & \text{ if } \; q\2=\20\,.
\end{cases}
\end{equation}

In particular, the generating density vanishes in either case on the corresponding interval by the definition of $\Theta$ and $\im m(\tau)=0$. 
On the other hand, from the formula \eqref{ideal cusp solution} for the leading order  term $\widehat \Theta_0$ to $\Theta$ in  \eqref{cusp solution for Theta} we see that the real valued $ \Theta$ is decreasing somewhere inside these gaps in the support of the generating density, because
\[
\Theta(\omega)
\,=\,  
 \widehat \Theta_0(\omega)\1(\11+\ord(\abs{\omega}^{1/3}))
\,=\, -\2C_1\2\mathrm{sign}(\omega)\1\abs{\omega}^{1/3}(\11+\ord(\abs{\omega}^{1/3}))\,,
\]
with a positive constant $C_1>0$.
Using \eqref{leading order of change} to write
\[
m(\tau+\omega)-m(\tau) \,=\,|\1m(\tau)|\2f(\tau)\2 \Theta(\omega;\tau)\1(\11+ \ord(\abs{\omega}^{1/3}))
\,,
\]
we see that correspondingly $m $  would have to decrease as well.
This contradicts the Stieltjes transform representation \eqref{m as stieltjes transform} of $m$, because the Stieltjes transform of a positive measure is monotonically increasing away from the support of that measure, when evaluated on the real line. 
Having ruled out the choices $ p=- $ and $ q=+$ earlier, we conclude that $p=+$ and $q=-$.

By \eqref{leading order of change} the function $\Theta$ describes the leading order of the difference between $m(\tau)$ to $m(\tau+\omega)$. Considering only the imaginary part and using the identity $\im m(\tau+\omega)= \pi \2v(\tau+\omega)$ we find
\begin{equation}
\label{v at the cusp}
v(\tau+\omega)\,=\, |m(\tau)|\2f(\tau) \im \Theta(\omega;\tau)+ \ord(\abs{\omega}^{2/3})
\,.
\end{equation}
From this \eqref{cusp singularity} follows if we define $c \in \Bounded(\xSpace, \R^+)$ by
\[
c_x\,:=\,\frac{\!\sqrt{3\2}}{2\pi}\pbb{\frac{\avg{\1|\1m(\tau)|\2f(\tau)\1}}{\psi(\tau)}}^{1/3}\! |\1m_x(\tau)|\2 f_x(\tau) \,.
\]
The behavior \eqref{v at the cusp} of the generating density around $\tau$ with $\sigma(\tau)=0$ shows that such a $\tau \in \partial\mathfrak{S}$ belongs to the intersection of the closure of two connected component of $\mathfrak{S}$, i.e., $\tau$ is a cusp in the sense of the statement of Theorem \ref{thr:Regularity and singularities of the generating density}. It also verifies \eqref{cusp singularity} at such an expansion point $\tau$.

\medskip
\noindent{\scshape Edge:}  Let $\sigma=\sigma(\tau)\neq 0$. We will show that $\tau$ is not a cusp. More precisely, we will show that with $\theta:= \mathrm{sign} \2\sigma$ the generating density satisfies
\begin{equation}
\label{v behavior at the edge}
v(\tau+ \theta\omega)\,=\, 
\begin{cases}
c\1 \abs{\omega}^{1/2}+\ord(\abs{\omega/\sigma^2}) 	& \text{ if }\; \theta\omega \in  [\10,c_3 \abs{\sigma}^3]\,,
\\
0								& \text{ if }\; \theta\omega \in  [-c_3 \abs{\sigma}^3,\10\1]\,,
\end{cases}
\end{equation}
for some $c \in \Bounded(\xSpace, \R^+)$ and some constant $c_3\sim 1$. 

We will again make use of Proposition \ref{prp:Cubic equation}. We write the corresponding cubic equation \eqref{cubic equation} in the form
\begin{equation}
\label{quadratic at the edge}
\sigma \1\Theta(\omega)^2 +\avg{|m|f}\1\omega
\,=\,
 e_3(\omega)\,,
\end{equation}
where the cubic term in $\Theta$ is considered part of the error,
\[
 e_3(\omega)\,:=\, e(\omega)- \psi\1\Theta(\omega)^3.
\]
With the bound \eqref{bounds on e} on $\abs{e(\omega)}$ and the a priori estimate $\abs{\Theta(\omega)}\lesssim \abs{\omega}^{1/3}$ (cf. \eqref{a priori Holder estimate}) we see from \eqref{quadratic at the edge} that $\Theta$ satisfies for some positive constant $c_4 \sim 1$ the bound
\begin{equation*}
\abs{\Theta(\omega)}\,\lesssim\, \abs{\omega/\sigma}^{1/2},\qquad \omega \in [-c_4\abs{\sigma}^3,c_4\abs{\sigma}^3]\,.
\end{equation*}
In particular, $\abs{ e_3(\omega)}\lesssim \abs{\omega/\sigma}^{3/2}$. We conclude that as a continuous solution of \eqref{quadratic at the edge} the function $\Theta$ has the form
\begin{equation}
\label{edge solution for Theta}
\Theta(\omega)\,=\, \widehat \Theta_{p}(\omega)\mathbbm{1}(\theta \omega < 0)+\widehat\Theta_{q}(\omega)\mathbbm{1}(\theta\omega \geq 0)+\ord(\abs{\omega/\sigma^2})\,,
\end{equation}
with $p,q \in \{+,-\}$.
Here, $\widehat \Theta_\#$ denotes the solution of \eqref{quadratic at the edge} with vanishing error term,
\begin{equation}
\label{ideal solution at the edge}
\widehat \Theta_\pm  (\omega)\,:=\, 
\pbb{\!\frac{\avg{\1|m|\1f\1}}{\abs{\sigma}}\!}^{\!1/2}\msp{-7}
\times
\begin{cases}
\pm\ii \abs{\omega}^{1/2},  &\;\theta \omega\geq 0\,,
\\
\pm\abs{\omega}^{1/2}, &\; \theta \omega<0\,.
\end{cases}
\end{equation}

First we notice that for the indices in \eqref{edge solution for Theta} the choice $q=-$ is impossible, because it violates $\im \Theta\geq 0$, as can be seen from
\[
\im \Theta(\omega)\,=\, \im \widehat \Theta_b(\omega)\pb{\11+\ord\pb{\abs{\omega/\sigma^3}^{1/2}}}\,,
\]
for $\theta \omega \geq0$.
Thus, the behavior \eqref{v behavior at the edge} of the generating density is proven for $\theta \omega \geq 0$ by taking the imaginary part of \eqref{leading order of change} and defining
\[
c_x\,:=\, \frac{1}{\pi}\pbb{\frac{\avg{|m(\tau)|f(\tau)}}{\abs{\sigma(\tau)}}}^{1/2}\! |\1m_x(\tau)|\2f_x(\tau) \,.
\]

Now we show that there is a gap in the support of the generating density for $\theta\omega \in  [-c_3 \abs{\sigma}^3,0]$, i.e., we show that $\im \Theta(\omega)$ vanishes for these values of $\omega$.  The explicit form of $ \widehat \Theta_\pm$ in \eqref{ideal solution at the edge} and \eqref{edge solution for Theta} reveal that 
\begin{equation}
\label{bounds on Theta at the edge}
\abs{\1\re {\Theta(\omega)}} \,\sim\, \abs{\omega/\sigma}^{1/2},\qquad \im \Theta(\omega)\,\lesssim\, \abs{\omega/\sigma^2}\,,
\end{equation}
as long as $\theta \omega\in [-c_5\abs{\sigma}^3,0] $ for some constant $c_5 \sim 1$.
Knowing the size of $\re \Theta$ from \eqref{bounds on Theta at the edge} we take the imaginary part of the quadratic equation \eqref{quadratic at the edge} and find
\begin{equation}
\label{bound on im Theta at the edge}
\abs{\sigma\omega}^{1/2} \im \Theta\,\lesssim\, \abs{\1\im (\Theta^3)}+\abs{\omega}\im \Theta
\,\lesssim\,\abs{\omega/\sigma}\im \Theta\,.
\end{equation}
Here, \eqref{bounds on e} was used in the first, and \eqref{bounds on Theta at the edge} as well as $\abs{\omega}\lesssim \abs{\sigma}^3 $ in the second inequality. 
 From \eqref{bound on im Theta at the edge} we conclude that $\im \Theta(\omega)=0$ for $\theta \omega\in [-c_3\abs{\sigma}^3,0] $ for some positive constant $c_3 \sim 1$. 

This finishes the proof of \eqref{v behavior at the edge}. In particular, we see that $\tau$ is not a cusp in the sense of the statement of Theorem \ref{thr:Regularity and singularities of the generating density}. We also conclude that \eqref{edge singularity} holds true at such an expansion point $\tau$, apart from the fact that the function $\abs{\omega/\sigma}$  inside the $\ord$-notation in \eqref{v behavior at the edge} still depends on the uncontrolled quantity $\abs{\sigma}$. This will be remedied by the fact that there can only be finitely many such points, because $\partial \mathfrak{S}$ is a finite set, as we will show below. We may then estimate the quantity  $\abs{\sigma}$ inside the error terms by the constant $\min\{\abs{\sigma(\tau)}: \tau \in \partial \mathfrak{S},\, \sigma(\tau)\neq 0\}>0$. 

In order to show that $ \partial \mathfrak{S} $ is finite we derive a contradiction by assuming the contrary. 
Since $ \mathfrak{S} $ is bounded (Proposition \ref{prp:Existence and uniqueness}) it follows that the closed infinite  set $ \partial \mathfrak{S} $ contains an accumulation point $\tau_\ast \in \partial \mathfrak{S}$.
If $ \sigma(\tau_\ast)=0 $, then  $ v(\tau_\ast +\omega) > 0 $ for every $ \abs{\omega} \in (0,\eps) $, and some $ \eps > 0$, by the already proven expansion \eqref{cusp singularity} for such points  $ \tau_\ast$. This contradicts $ \tau_\ast \in \partial  \mathfrak{S} $, because the generating density vanishes at every point of $\partial\mathfrak{S} $.
Thus, we have $ \sigma(\tau_\ast) \neq 0 $. 
Using the expansion \eqref{v behavior at the edge} at $ \tau = \tau_\ast $ we see that $ \tau_\ast $ is isolated from other elements of $\partial\mathfrak{S} $ by a distance $ c_3\1\abs{\sigma(\tau_\ast)}^3 > 0 $.
This contradicts $\tau_\ast $ being an accumulation point of $\partial\mathfrak{S} $ as well. Hence, we arrive at the conclusion that $ \partial \mathfrak{S} $ is finite.
This finishes the proof of our main result, Theorem~\ref{thr:Regularity and singularities of the generating density}.
\end{proof}

\begin{proof}[Proof of Theorem \ref{thr:Single interval support}] Suppose $\tau_0 \in \partial \mathfrak{S}$. In particular, $m(\tau_0)$ is real. We will first show that under the assumption \eqref{connected rows} on $S $ and $ a$ the solution $ m $ has  a definite sign at $ \tau_0 $. More precisely, we show that there exists $ \theta= \theta(\tau_0) \in \{\pm 1\} $ such that
\begin{equation}
\label{common sign}
\mathrm{sign}\,m_x(\tau_0) = \theta
\,,\qquad
\forall\1x \in \xSpace
\,.
\end{equation}
For the proof, we use the QVE to obtain
\begin{equation}
\label{m difference}
m_x(\tau)-m_y(\tau) 
\,=\, m_x(\tau)\2m_y(\tau)\2\pb{\,a_x-a_y+\avg{\1m\,(S_x-S_y)}\,}
\,,
\end{equation}
for every $x,y \in \xSpace$ and $\tau \in \R$.
Suppose now that \eqref{common sign} is not true, so that the set
\[
\mathfrak{A}\,:=\, \bigl\{x \in \xSpace: \; m_x(\tau_0)>0\bigr\}
\,.
\]
is not trivial. 
Choosing $ x \in \mathfrak{A}$, $ y\notin\mathfrak{A} $ and $ \tau = \tau_0 $ in \eqref{m difference} yields
\[
1 \,\lesssim\, m_x(\tau_0)-m_y(\tau_0)\,\lesssim\,\abs{a_x-a_y}+ \avg{\1\abs{S_x-S_y}\1}
\,,
\] 
where in the first estimate we used the lower bound $ \abs{m} \gtrsim 1 $ and in the second estimate the upper bound $ \abs{m} \lesssim 1 $ from \eqref{uniform bound}. 
Taking the infimum over $x \in \mathfrak{A} $ and $y \not \in \mathfrak{A} $ contradicts the assumption \eqref{connected rows}. 
Hence, we conclude that either $\mathfrak{A}=\emptyset$ or $\mathfrak{A}=\xSpace $, which is equivalent to \eqref{common sign}.

We will now show that $\tau_0$ is either the very right or the very left edge of $ \mathfrak{S}$, i.e., we prove that either the generating density vanishes on $[\tau_0, \infty)$ or on $(-\infty,\tau_0]$. Thus, $\partial \mathfrak{S}$ consists of only two points and $ \mathfrak{S}$ is a single interval.  

First we rule out the possibility that $\tau_0 $ is a cusp. 
Using \eqref{common sign} in the definition \eqref{definition of sigma and psi} of $ \sigma $ we conclude that $ \sigma(\tau_0) \neq 0 $ and $ \mathrm{sign}\, \sigma(\tau_0) = \theta $.
From \eqref{v behavior at the edge}, in the proof of Theorem \ref{thr:Regularity and singularities of the generating density}, we see that $ \tau_0 $ is not a cusp, and that there is a non-trivial connected component $ I $ of $ \R \backslash \mathfrak{S} $ containing $ \tau_0 $. The expansion \eqref{v behavior at the edge} also implies that  $ I $ continues in the direction $ -\theta $ from $ \tau_0 $. From here on we restrict the discussion to $\theta=-1$. The case $\theta=+1$ is treated analogously. 

We will now finish the proof by showing that $I=[\tau_0,\infty)$.  
By the continuity of $m(\tau)$ in $\tau$ and the lower bound $\abs{m}\gtrsim 1$ from \eqref{uniform bound} the sign of $m$ stays constant on the interval $I\cap[\tau_0,2\kappa]$. For $\tau>2\kappa$ we have $m(\tau)<0$ by \eqref{m as stieltjes transform}. Therefore, \eqref{common sign} extends to
\begin{equation}
\label{sign m = theta on I} 
\mathrm{sign}\,m(\tau) = -1\,,\qquad
\forall\,\tau \in I\,.
\end{equation}
\item
All the components $ \tau \mapsto m_x(\tau) $ are strictly increasing functions on $ I $.
This is a consequence of $ m_x $ being the Stieltjes transform (cf. \eqref{m as stieltjes transform}) of the non-negative density $ v_x $ which vanishes on $I$. 
Combining this with \eqref{sign m = theta on I} we deduce that  $ \tau \mapsto \abs{\1m_x(\tau)} $ is strictly decreasing for all $x \in \xSpace$. 
Decreasing $\abs{m}$ also decreases the spectral radius of the operator $F$, defined in  \eqref{definition of F}.
In particular, $ \norm{F(\tau)}_2 $ decreases strictly as $ \tau \in I $ is moved away from $ \tau_0 $. In particular,
\begin{equation}
\label{2-norm of F at tau0 and on I}
\norm{F(\tau_0)}_2 \,>\, \norm{F(\tau)}_2\,, 
\qquad\forall\,\tau \in I \backslash \{\tau_0\}
\,.
\end{equation}
Since from \eqref{spectral radius 1 at singularity} we know that $\norm{F(\tau)}_2 = 1 $ for any $\tau \in \partial \mathfrak{S}$, \eqref{2-norm of F at tau0 and on I} implies that $ I $ does not contain an element of $ \partial \mathfrak{S} $ other than $ \tau_0 $.
 This completes the proof of Theorem \ref{thr:Single interval support}.
\end{proof}

        

\appendix

\section{Existence and uniqueness}
\label{appendix:Existence and uniqueness}

Existence and uniqueness of \eqref{QVE} is established by interpreting the QVE as a fixed point equation for a holomorphic map in an appropriately chosen function space.  
The choice of the correct metric on this space follows naturally from the general theory by Earle and Hamilton \cite{EarleHamilton70}.  
The same line of reasoning has appeared before in a context close to ours in \cite{FHS2006,Helton2007-OSE,KLW2}.

\begin{proof}[Proof of Proposition \ref{prp:Existence and uniqueness}]
We will set up the fixed point problem on the set of functions defined on the domain
\[
\Cp_\eta\,:=\, \bigl\{z \in \Cp\,:\; \im z \2\geq\2 \eta\,,\; \abs{z}\2\leq\2 \eta^{-1}\bigr\}\,.
\]
More precisely, for any $\eta \in (0,\min\{1,\norm{a}^{-1}\})$ we consider the function space 
\begin{equation}
\label{definition of Beta}\!
\mathfrak{B}_{\msp{-1}\eta} 
:=
\biggl\{ \mathfrak{u}: \Cp_{\1\eta} \to \Bounded(\xSpace,\Cp) :\inf_{z \in \Cp_\eta}\!\im \mathfrak{u}(z)\geq \frac{\eta^3}{(2+\norm{S})^2},\,  \sup_{z\in \Cp_{\1\eta}}\norm{\mathfrak{u}(z)}\leq \frac{1}{\eta} \biggr\}
,
\end{equation}
equipped with the metric 
\[
d_\mathfrak{B}(\mathfrak{u},\mathfrak{w})\,:=\, \sup_{z \in \Cp_{\1\eta}}\sup_{x \in \xSpace}d_{\Cp}(\mathfrak{u}_x(z),\mathfrak{w}_x(z))\,, \qquad \mathfrak{u}, \mathfrak{w} \in \mathfrak{B}_{\msp{-1}\eta} \,,
\]
where $d_{\Cp}$ denotes the standard hyperbolic metric on $\Cp$. The metric function space $(\mathfrak{B}_{\msp{-1}\eta}, d_\mathfrak{B})$ is complete.  In this setting the QVE takes the form 
\begin{equation}
\label{QVE as fixed point equation}
\mathfrak{u}\,=\, \Phi( \mathfrak{u})\,,
\end{equation}
where the function $\Phi$ is defined as 
\begin{equation}
\label{definition of Phi}
\Phi( \mathfrak{u})(z)\,:=\, -\2\frac{1}{z+a+S\1\mathfrak{u}(z)}\,,\qquad \mathfrak{u} \in \mathfrak{B}_{\msp{-2}\eta}\,,\; z \in \Cp_{\1\eta}\,.
\end{equation}

We will now verify that $\Phi$ is well defined as a map from $ \mathfrak{B}_{\msp{-1}\eta}$ to itself. In fact, $\Phi$, defined as in \eqref{definition of Phi}, maps all functions $ \mathfrak{u}: \Cp \to \Bounded(\xSpace,\Cp)$ to functions with the upper bound
\begin{equation}
\label{trivial upper bound for Phi}
\abs{\1\Phi(\mathfrak{u})(z)}\,=\, \frac{1}{\abs{z+ a+S\1\mathfrak{u}(z)}}\,\leq\, \frac{1}{\im (z+ a+S\1\mathfrak{u}(z))}\,\leq\, \frac{1}{\im \1 z}\,,
\end{equation}
where we used that $S$ has a non-negative kernel and therefore $\im S\1\mathfrak{u}(z) \geq 0$. Taking the supremum over all $z \in \Cp_{\1\eta}$ in \eqref{trivial upper bound for Phi} reveals that $\sup_{z \in \Cp_\eta}\norm{\Phi(\mathfrak{u})(z)}\leq \eta^{-1}$, which is the upper bound in the definition \eqref{definition of Beta} of $\mathfrak{B}_{\msp{-1}\eta}$.  

On the other hand, for every function  $\mathfrak{u}: \Cp_{\1\eta} \to \Bounded(\xSpace,\Cp)$ that satisfies the upper bound $\sup_{z\in \Cp_{\1\eta}}\norm{\mathfrak{u}(z)}\2\leq\2 \eta^{-1}$, we find a lower bound on the imaginary part of $\Phi(\mathfrak{u})$,
\begin{equation}
\label{lower bound on im Phi}
\im \Phi(\mathfrak{u})(z)= \frac{\im (z+ a+S\1\mathfrak{u}(z))}{\abs{z+ a+S\1\mathfrak{u}(z)}^2}\geq 
\frac{\im z}{(\abs{z}+\norm{a}+\norm{S} \eta^{-1})^2}\geq \frac{\eta^3}{(2+\norm{S})^2}\,, 
\end{equation}
for every $z \in \Cp_{\1\eta}$. 
Thus, $\Phi:\mathfrak{B}_{\msp{-1}\eta} \to \mathfrak{B}_{\msp{-1}\eta}$ is well defined. 

The two computations in \eqref{trivial upper bound for Phi} and \eqref{lower bound on im Phi} also show that the restriction $m|_{\Cp_\eta}$ of any solution to the QVE automatically belongs to $\mathfrak{B}_{\msp{-1}\eta}$. In particular, showing existence and uniqueness of the solution $\mathfrak{u}$ to the fixed point equation \eqref{QVE as fixed point equation} on $\mathfrak{B}_{\msp{-1}\eta}$ for every positive $\eta<\min\{1,\norm{a}^{-1}\}$ is equivalent to showing existence and uniqueness of the solution $m: \Cp \to \Bounded(\xSpace,\Cp)$ to the QVE. 

We will now establish a certain contraction property of the map $\Phi$. This property is expressed in terms of the function
\begin{equation}
\label{definition of metric D}
D(\zeta,\omega)\,:=\,\frac{\abs{\1\zeta-\omega}^2}{(\im\,\zeta)(\im\,\omega)}\,,
\qquad  \zeta, \,\omega \in \Cp
\,,
\end{equation}
which is related to the standard hyperbolic metric $d_\Cp$ by
\begin{equation}
\label{D and hyperbolic metric}
D(\zeta,\omega)\,=\,2\2(\cosh d_\Cp(\zeta,\omega)-1)\,.
\end{equation}

\begin{lemma}[Contraction property of $\Phi$]
\label{lmm:Contraction property}
For any  $\mathfrak{u}, \mathfrak{w} \in \mathfrak{B}_{\msp{-1}\eta}$ the map $\Phi:\mathfrak{B}_{\msp{-1}\eta} \to \mathfrak{B}_{\msp{-1}\eta}$ has the contraction property
\begin{equation}
\label{contraction property}
\sup_{ z\in \Cp_{\eta}}\sup_{x \in \xSpace}D(\Phi( \mathfrak{u})_x(z),\Phi( \mathfrak{w})_x(z))\,\leq\, \bigg(1+\frac{\,\eta^2}{\norm{S}}\bigg)^{-2}\msp{-5}\sup_{ z\in \Cp_{\eta}}\sup_{x \in \xSpace}D(\mathfrak{u}_x(z),\mathfrak{w}_x(z))\,.
\end{equation}
In particular, the fixed point equation \eqref{QVE as fixed point equation} has a unique solution $ \mathfrak{u} \in \mathfrak{B}_{\msp{-1}\eta}$.
\end{lemma}

We postpone the proof of \eqref{contraction property}, and with it the proof of Lemma \ref{lmm:Contraction property}, until after the end of the proof of Proposition \ref{prp:Existence and uniqueness}. The contraction property \eqref{contraction property} shows that for any initial value $\mathfrak{u}^{(0)} \in \mathfrak{B}_{\msp{-1}\eta}$, the sequence of iterates $\mathfrak{u}^{(k)}:= \Phi^k(\mathfrak{u}^{(0)})$ is a Cauchy-sequence in $\mathfrak{B}_{\msp{-1}\eta}$. Therefore, $\mathfrak{u}^{(k)}$ converges to the unique fixed point of $\Phi$ and thus to the restriction $m |_{\mathfrak{B}_{\eta}}$ of the unique solution $m$ to the QVE.

If we start the iteration  $\Phi^k(\mathfrak{u}^{(0)})$ with a choice of $\mathfrak{u}^{(0)}$ that is continuous in $z$ and holomorphic in the interior of $\Cp_\eta$ (e.g. $\mathfrak{u}^{(0)}_{\2x}(z):= \ii$), then every iterate has this property. Since the space of such functions is a closed subspace of  $\mathfrak{B}_{\msp{-1}\eta}$, the limit $m |_{\mathfrak{B}_{\eta}}$ is holomorphic in the interior of $\Cp_\eta$.  Since $\eta$ was arbitrary, we conclude that the unique solution $m(z)$ of the QVE is a holomorphic function of the spectral parameter $z \in \Cp$ on the entire complex upper half plane. 

Now we show the representation \eqref{m as stieltjes transform} for $m(z)$. 
We use that a holomorphic function on the complex upper half plane $ \phi:~\Cp~\to~\Cp$ is a Stieltjes transform of a probability measure on the real line if and only if $\abs{\ii\eta  \phi(\ii  \eta)+1} \to 0$ as $ \eta \to \infty$ (cf. Theorem~3.5 in \cite{Garnett-BA2007}, for example).
In order to see that
\begin{equation}
\label{condition for S-transform representation}
\lim_{\eta \to \infty}\sup_{x\in \xSpace}\,\absb{\ii \eta \1m_x(\ii\eta)+1} \,=\,0\,, 
\end{equation}
we write the QVE in the quadratic form
\[
z\2 m(z)+1 \,=\, -a \2m(z)-m(z)\1Sm(z)
\,.
\]
From this we obtain
\[
\abs{z\2m(z)+1} \,\leq\,\norm{a}\norm{m(z)}+\norm{S} \norm{m(z)}^2
\,.
\]
The right hand side is bounded by using \eqref{trivial upper bound for Phi} with $\mathfrak{u}:=m$:
\begin{equation}
\label{m trivial bound}
\abs{m(z)} \,\leq \, \frac{1}{\im z}\,, \qquad 
 z \in \Cp\,.
\end{equation}
Choosing $ z = \ii\eta $, we get $ \abs{\ii  \eta \1m(\ii\eta)+1} \leq \norm{a} \eta^{-1}+ \norm{S} \eta^{-2} $, and hence \eqref{condition for S-transform representation} holds true. 
This completes the proof of the Stieltjes transform representation \eqref{m as stieltjes transform}.

To finish the proof of Proposition \ref{prp:Existence and uniqueness} we show that the support of the $x$-th component $v_x$ of the generating measure lies for all $x$ in the common compact interval $[-\kappa,\kappa]$ with $\kappa$ defined in \eqref{definition of kappa}. In fact, from the Stieltjes transform representation \eqref{m as stieltjes transform} it suffices to show that $\im m_x(\tau +\ii \eta)$ converges to zero locally uniformly for $\abs{\tau}> \kappa$ as $\eta \downarrow 0$. 

From the QVE we read off that for any $z \in \Cp$ with $\abs{ z}>\kappa$ the following implication holds:
\[
\text{If } \; \norm{m(z)}\,<\, \frac{\abs{z}-\norm{a}}{2\1\norm{S}}\,,\; \text{ then }\;\norm{m(z)} \,<\, \frac{1}{\abs{z}-\norm{a}-\norm{S}\norm{m(z)}}\,\leq\, \frac{2}{\abs{z}-\norm{a}}\,.
\]
In particular, we see that there is a gap in the values that $\norm{m}$ can take,
\[
\norm{m(z)}\not \in \sB{\frac{2}{\abs{z}-\norm{a}},\frac{\abs{z}-\norm{a}}{2\1\norm{S}}}\,,\qquad \abs{z}>\kappa\,.
\]
Since $z \mapsto \norm{m(z)}$ is a continuous function and by \eqref{m trivial bound} the value of $ \norm{m(z)}$ lies below this gap for large values of $\im z$, we conclude that 
\begin{equation}
\label{m bound for large z}
\norm{m(z)}\,\leq\, \frac{2}{\abs{z}-\norm{a}}\,,\qquad \abs{z}>\kappa\,.
\end{equation}

Now we consider the imaginary part of the QVE,
\[
\frac{\im m(z)}{\,\abs{m(z)}^2\!}\,=\, -\im \frac{1}{m(z)}\,=\, \im z + S \im m(z)\,.
\]
We take the norm on both sides of this equation and use the bound \eqref{m bound for large z} to see that
\[
\norm{\im m(z)}\,\leq\, 4\, \frac{\im z + \norm{S}\norm{\im m(z)}}{(\abs{z}-\norm{a})^2}\,,\qquad \abs{z}>\kappa\,.
\]
By the definition \eqref{definition of kappa} of $\kappa$ the coefficient in front of $\norm{\im m}$ on the right hand side is smaller than $1$. Thus we end up with an upper bound on the imaginary part of the solution,
\[
\norm{\im m(z)}\,\leq\, \frac{4\2\im z}{(\abs{z}-\norm{a})^2-4\norm{S}}\,,\qquad \abs{z}>\kappa\,.
\]
In particular, this bound shows that 
\[
\lim_{\eta \downarrow 0} \sup_{\abs{\tau} \1\geq\1 \kappa+\eps}\! \norm{\1\im m(\tau +\ii \eta)}\,= \, 0\,,
\]
for any $\eps>0$. This finishes the proof of Proposition \ref{prp:Existence and uniqueness}.
\end{proof}

\begin{proof}[Proof of Lemma \ref{lmm:Contraction property}]
We show that, more generally than \eqref{contraction property}, for all functions $\mathfrak{u}, \mathfrak{w}: \Cp \to \Bounded(\xSpace, \Cp)$ and all $z \in \Cp$ we have
\begin{equation}
\label{contraction property at z}
D(\Phi( \mathfrak{u})_x(z),\Phi( \mathfrak{w})_x(z))
\,\leq\, 
\bigg(1+\frac{(\im z)^2}{\norm{S}}\bigg)^{-2}\msp{-5} \sup_{y \1\in\1 \xSpace}D(\mathfrak{u}_{\1y}(z),\mathfrak{w}_{y}(z))\,.
\end{equation}
To see this we need the following properties of the function $D$.  

\begin{lemma}[Properties of hyperbolic metric]
\label{lmm:Properties of hyperbolic metric}
The following three properties hold for $D$:
\begin{enumerate}
\item \label{cProp1}
Isometries: If $ \psi : \Cp \to \Cp $, is a M\"obius transform, of the form 
\[
\psi(\zeta) \,=\,\frac{\alpha \1 \zeta+\beta}{\gamma \1 \zeta+\delta}
\,,\qquad
\left(
\begin{array}{cc}
\alpha & \beta 
\\ 
\gamma &\delta
\end{array}
\right)
\in \mathrm{SL}_2(\R)\,, 
\]
then 
\[
D\big(\psi(\zeta),\psi(\omega)\big) \,=\, D(\zeta,\omega) 
\,.
\]
\item \label{cProp2}
Contraction: If $ \zeta $, $\omega \in \Cp $ are shifted in the positive imaginary direction by $ \eta> 0$ then  
\begin{equation*}
D(\1\ii \eta +\zeta, \ii \eta + \omega)
\,=\,
\Big(1+\frac{\eta}{\im\,\zeta}\Big)^{-1}
\Big(1+\frac{\eta}{\im\,\omega}\Big)^{-1}D(\zeta,\omega)
\,.
\end{equation*}
\item \label{cProp3} 
Convexity: 
Let $\phi\neq 0$ be a non-negative bounded linear functional on $\Bounded(\xSpace,\C)$, i.e., $\phi(w)\geq 0$ for all $w \geq 0$. Then
\[
D(\phi(w),\phi(u))\,\leq\, \sup_{x \in \xSpace} \, D(w_x,u_x)\,,
\]
for all $w,u \in \Bounded(\xSpace,\Cp)$ with $\inf_{x \in \xSpace}\im w_x>0$ and $\inf_{x \in \xSpace}\im u_x>0$. 
\end{enumerate}
\end{lemma}
The properties (\ref{cProp1}) and (\ref{cProp2}) are  clear from the connection \eqref{D and hyperbolic metric} of $D$ to the hyperbolic metric and the definition of $D$ in \eqref{definition of metric D}, respectively. For a short proof of the property (\ref{cProp3}) in the setup where $\xSpace$ is finite, we refer to Lemma 5 in  \cite{KLW2}. Since that argument can easily be adapted to the case of general $\xSpace$, we will omit the proof.

In case $ S_x \neq 0 $ we use Lemma \ref{lmm:Properties of hyperbolic metric} and the definition of $\Phi$ in \eqref{definition of Phi} to estimate
\begin{equation*}
\begin{split}
D(\Phi( \mathfrak{u})_x(z),\Phi( \mathfrak{w})_x(z))
\,&=\, D\bigl(\2z + \avg{S_x\2\mathfrak{u}(z)}\2,\,z + \avg{S_x\2\mathfrak{w}(z)} \bigr)
\\
\,&\leq\, 
(1+(\im z)^2/\norm{S})^{-2}\, D\bigl(\1\re z +  \avg{\1S_x\2\mathfrak{u}(z)},\,\re z+ \avg{\1S_x\2\mathfrak{w}(z)}\bigr)
\\
\,&=\, 
(1+(\im z)^2/\norm{S})^{-2}\,D\bigl( \avg{\1S_x\2\mathfrak{u}(z)} , \avg{\1S_x\2\mathfrak{w}(z)}\bigr)
\\
\,&\leq\, 
(1+(\im z)^2/\norm{S})^{-2}\,\sup_{y \in \xSpace} \,D(\mathfrak{u}_{\1y}(z),\mathfrak{w}_{y}(z))
\,.
\end{split}
\end{equation*}
For the equalities we applied property (\ref{cProp1}). In the first inequality property (\ref{cProp2}) and in the second inequality property (\ref{cProp3}) was used.
On the other hand, if $ S_x = 0 $ then the claim is trivial. 

This finishes the proof of Lemma \ref{lmm:Contraction property}. The existence and uniqueness for the fixed point problem \eqref{QVE as fixed point equation} follows as explained after the statement of the lemma. 
\end{proof}

\section{Auxiliary results}
\begin{proof}[Proof of Lemma \ref{lmm:Spectral gap for operators with positive kernel}] In case $\inf_{x,y \in \xSpace}t_{x y}=0$, there is nothing to show. Thus, we assume $\eps:=\inf_{x,y \in \xSpace}t_{x y}>0$. 
Without loss of generality we may assume that $ \norm{T}_2 = 1 $ and that $ h $ is the unique eigenfunction satisfying $\norm{h}_2=1 $ and $ h \ge 0 $.
First we note that
\[
h \,=\, Th \,\geq\, \eps \int  h_x \2\xMeasure(\dd x) \,>\,0
\,.
\]
Since the kernel of $T$ is real, we work on the space of real valued functions in $L^{2}(\xSpace) $. 
Evaluating the quadratic form of $1\pm T$ at some $ u $ that is orthogonal to $h$ yields
\begin{equation*}
\begin{split}
\avgb{u,(1\pm\1T)u}\;&=\;\frac{1}{2}\iint  t_{x y}\,
\biggl(
u_x\,\sqrt{\frac{h_y}{h_x}}\,\pm\, u_y\,\sqrt{\frac{h_x}{h_y}}
\;\biggr)^{\!2}\xMeasure(\dd y)\,\xMeasure(\dd x)
\\
&\geq\; 
\frac{\eps}{2\1\norm{h}^2\!} 
\iint
h_x\, h_y \biggl(u_x^2\; \frac{h_y}{h_x}\,+\,u_y^2\;\frac{h_x}{h_y}\,\pm\, 2 \,u_x\, u_y\biggr)
\,\xMeasure(\dd y)\,\xMeasure(\dd x)
\\
&=\;
\frac{\eps}{\norm{h}^2\!} \int u_x^2\, \xMeasure(\dd x)
\;,
\end{split}
\end{equation*}
where in the inequality we used $ t_{x y}\geq \eps\geq \eps\2h_x h_y/\norm{h}^2$ for almost all $x,y \in \xSpace$. Now we read off that
\begin{equation*}
-\,\Bigl(1-\frac{\eps}{\norm{h}^2\!}\,\Bigr)\norm{u}_2^2\,\leq\, \scalar{\1u}{T u\1}\,\leq\, \Bigl(\11-\frac{\eps}{\norm{h}^2\!}\,\Bigr)\,\norm{u}_2^2
\,.
\end{equation*}
This shows the gap in the spectrum of the operator $T$.
\end{proof}

\begin{proof}[Proof of Lemma \ref{lmm:Norm of B^-1-type operators on L2}]
Proving the claim \eqref{bound on U-T inverse} is equivalent to proving that
\begin{equation}
\label{Norm of B^-1-type operators on L2 - with w}
\norm{(U-T)w}_2
\,\ge\, 
c\,\alpha\2\mathrm{Gap}(T)\norm{w}_2 
\,,\qquad \alpha\2:=\2\abs{\11 - \norm{T}_2\scalar{\1h}{ Uh\1}}\,,
\end{equation}
for all $w \in L^2(\xSpace)$ and for some numerical constant $c>0$.  
To this end, let us fix $w$ with $\norm{w}_2=1$. We decompose $w$ according to the spectral projections of $T$,
\begin{equation}
\label{decomposition of w}
w \,=\, \scalar{h}{w}\2 h + Pw\,, 
\end{equation}
where $P$ is the projection onto the orthogonal complement of $t$. During this proof we will omit the lower index $2$ of all norms, since every calculation is in $L^{2}(\xSpace) $.
We will show the claim in three separate regimes: 
\begin{itemize}
\item[(i)] 
$ 16\2\norm{Pw}^2 \2\geq \2 \alpha$,
\item[(ii)] 
$ 16\2\norm{Pw}^2 \2< \2  \alpha$ 
and $\alpha\2\geq\2 \norm{PUh}^2$,
\item[(iii)] 
$ 16\2\norm{Pw}^2 \2< \2 \alpha$ 
and $ \alpha\2<\2 \norm{PUh}^2$.
\end{itemize}

In the regime (i) the triangle inequality yields
\[
\norm{(U-T)w}\,\geq\, \norm{w}- \norm{Tw} \,=\, 1- \big(\abs{\scalar{h}{w}}^2\2\norm{T}^2 + \norm{TPw}^2 \big)^{1/2}.
\]
We use the simple inequality, $1-\sqrt{1-\tau\2}\geq \tau/2$, valid for every $\tau  \in [0,1]$, and find
\begin{equation}
\begin{split}
2\2\norm{(U-T)w}\,
&\ge\, 
1 - \abs{\scalar{h}{w}}^2\1\norm{T}^2 - \norm{TP\1w}^2
\\
&\geq\, 
1 \,- \abs{\scalar{h}{w}}^2\1\norm{T}^2 \,-\, (\1\norm{T}-\mathrm{Gap}(T))^2\norm{Pw}^2
\\
&=\, \,1\,- \norm{T}^2 \,+\, (2\1 \norm{T}  - \mathrm{Gap}(T)\1)\2\mathrm{Gap}(T)\2\norm{Pw}^2
\,.
\end{split}
\end{equation}
The definition of the first regime implies the desired bound \eqref{Norm of B^-1-type operators on L2 - with w}.

In the regime (ii) we project the left hand side of \eqref{Norm of B^-1-type operators on L2 - with w} onto the $h$-direction,
\begin{equation}
\label{regime ii step 1}
\norm{(U-T)w}\,=\, \norm{(1-U^*T)w}\,\geq\, |\scalar{h}{(1-U^*T)w}|
\,.
\end{equation} 
Using the decomposition \eqref{decomposition of w} of $w$ and the orthogonality of $h$ and $Pw$, we estimate further:
\begin{equation}
\label{regime ii step 2}
\begin{split}
|\scalar{h}{(1-U^*T)w}|\,&\geq\,
\abs{\scalar{h}{w}}\abs{1-\norm{T}\scalar{h}{U^*t}} - \abs{\scalar{h}{U^*TPw}}
\\
\,&\geq\,\abs{\scalar{h}{w}} \alpha-\norm{PUh}\norm{Pw}
\,.
\end{split}
\end{equation}
Since $\alpha\leq 2$ and by the definition of the regime (ii) we have $\abs{\scalar{h}{w}}^2=1-\norm{Pw}^2 \geq 1-\alpha/16\geq 7/8$ and $\norm{PUh}\norm{Pw}\leq  \alpha/4$. Thus, we can combine \eqref{regime ii step 1} and \eqref{regime ii step 2} to 
\[
\norm{(U-T)w}\,\geq\, \alpha/2
\,.
\]

Finally, we treat the regime (iii). Here, we project the left hand side of \eqref{Norm of B^-1-type operators on L2 - with w} onto the orthogonal complement of $h$ and get
\begin{equation}
\label{regime iii step 1}
\norm{(U-T)w}
\,\ge\, 
\norm{P(U-T)w}
\,\geq\, 
\abs{\scalar{h}{w}}\norm{PUh} -\norm{P(U-T)Pw}
\,,
\end{equation}
where we inserted the decomposition \eqref{decomposition of w} again. 
In this regime we still have $\abs{\scalar{h}{w}}^2\geq 7/8$, and we continue with
\begin{equation}
\label{regime iii step 2}
\abs{\scalar{h}{w}}\norm{PUh} -\norm{P(U-T)Pw}\,\geq\, \frac{3}{4}\norm{PUh}- 2\norm{Pw}\,\geq\, \frac{\alpha^{1/2}\!}{2}.
\end{equation}
In the last inequality we used the definition of the regime (iii). 
Combining \eqref{regime iii step 1} with \eqref{regime iii step 2} yields
\[
\norm{(U-T)w}\,\geq\, \alpha/4
 \,,
\]
after using $ \norm{h}=1$ in \eqref{Norm of B^-1-type operators on L2 - with w} to estimate $ \alpha \leq 2 $. 
\end{proof}


\begin{thebibliography}{10}

\bibitem{AC-yau-students}
Adlam, B.; Che, Z. {Spectral Statistics of Sparse Random Graphs with a General
  Degree Distribution}. \emph{arXiv:1509.03368} .

\bibitem{AM}
Adler, M.; van Moerbeke, P. {PDEs for the Gaussian ensemble with external
  source and the Pearcey distribution}. \emph{Comm. Pure Appl. Math.}
  \textbf{60} (2007), no.~9, 1261--1292.

\bibitem{AEK3}
Ajanki, O.; Erd{\H o}s, L.; Kr\"uger, T. {Local spectral statistics of Gaussian matrices with correlated entries}. \emph{J. Stat. Phys.} \textbf{163} (2016), no.~2, 280-302 


\bibitem{AEK1}
Ajanki, O.; Erd{\H o}s, L.; Kr\"uger, T. {Quadratic vector equations on the
  complex upper half plane}. \emph{arXiv:1506.05095}.

\bibitem{AEK2}
Ajanki, O.; Erd{\H o}s, L.; Kr\"uger, T. {Universality for general Wigner-type
  matrices}. \emph{arXiv:1506.05098}.

\bibitem{IntroductionRM}
Anderson, G.; Guionnet, A.; Zeitouni, O. \emph{An Introduction to Random
  Matrices}, \emph{Cambridge Studies in Advanced Mathematics}, vol. 118,
  Cambridge University Press, 2010.

\bibitem{AZdep}
Anderson, G.; Zeitouni, O. {A Law of Large Numbers for Finite-Range Dependent
  Random Matrices}. \emph{Comm. Pure Appl. Math.} \textbf{61} (2008), no.~8,
  1118--1154.

\bibitem{AZind}
Anderson, G.~W.; Zeitouni, O. {A CLT for a band matrix model}. \emph{Probab.
  Theory Related Fields} \textbf{134} (2005), no.~2, 283--338.

\bibitem{BMP2013}
Banna, M.; Merlev\`ede, F.; Peligrad, M. {On the limiting spectral distribution
  for a large class of random matrices with correlated entries}. \emph{Stoch.
  Proc. Appl.} \textbf{125} (2015), no.~7, 2700--2726.

\bibitem{Ber}
Berezin, F. {Some remarks on Wigner distribution}. \emph{Theoret. Math. Phys.}
  \textbf{3} (1973), no.~17, 1163--1175.

\bibitem{BK}
Bleher, P.~M.; Kuijlaars, A. B.~J. {Large $n$ Limit of Gaussian Random Matrices
  with External Source, Part III: Double Scaling Limit}. \emph{Comm. Math.
  Phys.} \textbf{270} (2007), 481--517.

\bibitem{BolleNeri}
Boll{\'e}, D.; Metz, F.~L.; Neri, I. {On the spectra of large sparse graphs
  with cycles}. in \emph{Spectral analysis, differential equations and
  mathematical physics: a festschrift in honor of Fritz Gesztesy's 60th
  birthday}, pp. 35--58, Amer. Math. Soc., Providence, RI, 2013.

\bibitem{BH}
Br\'ezin, E.; Hikami, S. {Universal singularity at the closure of a gap in a
  random matrix theory}. \emph{Phys. Rev. E} \textbf{57} (1998), no.~4,
  4140--4149.

\bibitem{BurdaVARMA}
Burda, Z.; Jarosz, A.; Nowak, M.~A.; Snarska, M. {A random matrix approach to
  VARMA processes}. \emph{New Journal of Physics} \textbf{12} (2010), no.~7,
  075\,036.

\bibitem{Chakrabarty2014}
Chakrabarty, A.; Hazra, R.~S.; Sarkar, D. {From random matrices to long range
  dependence}. \emph{arXiv:1401.0780} .

\bibitem{DKMcL}
Deift, P.; Kriecherbauer, T.; McLaughlin, K. T.-R. {New Results on the
  Equilibrium Measure for Logarithmic Potentials in the Presence of an External
  Field}. \emph{J. Approx. Theory} \textbf{95} (1998), no.~3, 388--475.

\bibitem{EarleHamilton70}
Earle, C.~J.; Hamilton, R.~S. {A Fixed Point Theorem for Holomorphic Mappings}.
  \emph{Proc. Sympos. Pure Math.} \textbf{XVI} (1970), 61--65.

\bibitem{EYBull}
Erd{\H{o}}s, L.; Yau, H.-T. {Universality of local Spectral statistics of
  random matrices}. \emph{Bull. Amer. Math. Soc} \textbf{49} (2012), 377--414.

\bibitem{EYY}
Erd{\H o}s, L.; Yau, H.-T.; Yin, J. {Bulk universality for generalized Wigner
  matrices}. \emph{Probab. Theory Related Fields} \textbf{154} (2011), no. 1-2,
  341--407.

\bibitem{ErgunModGraph}
Erg{\"u}n, G.; K{\"u}hn, R. {Spectra of modular random graphs}. \emph{J. Phys.
  A} \textbf{42} (2009), no.~39, 395\,001--395\,014.

\bibitem{FHS2006}
Froese, R.; Hasler, D.; Spitzer, W. {Absolutely Continuous Spectrum for the
  Anderson Model on a Tree: A Geometric Proof of Klein's Theorem}. \emph{Comm.
  Math. Phys.} \textbf{269} (2007), no.~1, 239--257.

\bibitem{Garnett-BA2007}
Garnett, J. \emph{{Bounded Analytic Functions}}, \emph{Grad. Texts in Math.},
  vol. 236, Springer, New York, 2007.

\bibitem{Girko-book}
Girko, V.~L. \emph{{Theory of stochastic canonical equations. Vol. I}},
  \emph{Mathematics and its Applications}, vol. 535, Kluwer Academic
  Publishers, Dordrecht, 2001.

\bibitem{Guionnet-GaussBand}
Guionnet, A. {Large deviations upper bounds and central limit theorems for
  non-commutative functionals of Gaussian large random matrices}. \emph{Annales
  de l'IHP Probabilit{\'e}s et statistiques} \textbf{38} (2002), 341--384.

\bibitem{HachemCovSurvey2015}
Hachem, W.; Hardy, A.; Najim, J. {A survey on the eigenvalues local behaviour
  of large complex correlated Wishart matrices}. \emph{ESAIM: Proceedings and
  Surveys} \textbf{50} (2015), 150--174.

\bibitem{Helton2007-OSE}
Helton, J.~W.; Far, R.~R.; Speicher, R. {Operator-valued Semicircular Elements:
  Solving A Quadratic Matrix Equation with Positivity Constraints}. \emph{Int.
  Math. Res. Notices} \textbf{2007} (2007).

\bibitem{KLW2}
Keller, M.; Lenz, D.; Warzel, S. {On the spectral theory of trees with finite
  cone type}. \emph{Israel J. Math.} \textbf{194} (2013), no.~1, 107--135.

\bibitem{KhorunzhyPastur94}
Khorunzhy, A.~M.; Pastur, L.~A. {On the eigenvalue distribution of the deformed
  Wigner ensemble of random matrices}. in \emph{Spectral operator theory and
  related topics}, pp. 97--127, Adv. Soviet Math., 19, Amer. Math. Soc.,
  Providence, RI, 1994.

\bibitem{KuhnModGraph}
K{\"u}hn, R.; van Mourik, J. {Spectra of modular and small-world matrices}.
  \emph{J. Phys. A} \textbf{44} (2011), no.~16, 165\,205--165\,218.

\bibitem{KMcL}
Kuijlaars, A. B.~J.; McLaughlin, K. T.-R. {Generic behavior of the density of
  states in random matrix theory and equilibrium problems in the presence of
  real analytic external fields}. \emph{Comm. Pure Appl. Math.} \textbf{53}
  (2000), no.~6, 736--785.

\bibitem{Menon-IEEEplots}
Menon, R.; Gerstoft, P.; Hodgkiss, W.~S. {Asymptotic Eigenvalue Density of
  Noise Covariance Matrices}. \emph{IEEE Trans. Signal Process.} \textbf{60}
  (2012), no.~7, 3415--3424.

\bibitem{Mu}
Muskhelishvili, N.~I. \emph{{Singular Integral Equations: Boundary Problems of
  Function Theory and Their Application to Mathematical Physics}}, Courier
  Dover Publications, 2008.

\bibitem{PasturDefWig}
Pastur, L. {On the Spectrum of Random Matrices}. \emph{Theor. Math. Phys.}
  \textbf{10} (1972), no.~1, 67--74.

\bibitem{PasturSpecProb}
Pastur, L.~A. {Spectral and Probabilistic Aspects of Matrix Models}. in
  \emph{Algebraic and geometric methods in mathematical physics}, pp. 207--242,
  Math. Phys. Stud., 19, Kluwer Acad. Publ., Dordrecht, 1996.

\bibitem{PasturShcerbinaAMSbook}
Pastur, L.~A.; Shcherbina, M. \emph{{Eigenvalue Distribution of Large Random
  Matrices}}, \emph{Mathematical Surveys and Monographs}, vol. 171, Amer. Math.
  Soc., 2011.

\bibitem{RaoPolyDOS}
Rao, N.~R.; Edelman, A. {The Polynomial Method for Random Matrices}.
  \emph{Found. Comput, Math.} \textbf{8} (2007), no.~6, 649--702.

\bibitem{ShlyakhtenkoGBM}
Shlyakhtenko, D. {Random Gaussian band matrices and freeness with
  amalgamation}. \emph{Int. Math. Res. Notices}  (1996), no.~20, 1013--1015.

\bibitem{ChoiSiverstein1995}
Silverstein, J.~W.; Choi, S.~I. {Analysis of the limiting spectral distribution
  of large dimensional random matrices}. \emph{J. Multivar. Anal.} \textbf{54}
  (1995), 295--309.

\bibitem{TW}
Tracy, C.~A.; Widom, H. {The Pearcey Process}. \emph{Comm. Math. Phys.}
  \textbf{263} (2006), 381--400.

\bibitem{WegnerNorb}
Wegner, F.~J. {Disordered system with $ n $ orbitals per site: $ n =\infty $
  limit}. \emph{Phys. Rev. B} \textbf{19} (1979).


\end{thebibliography}

\end{document}